\documentclass[10pt,a4paper]{article}
\title{{Existence and Asymptotic Behavior of Solutions to a Quasilinear Hyperbolic-Parabolic Model of Vasculogenesis}}
\date{}
\author{Cristiana Di Russo$^{1,3}$, Alice Sepe$^{2,3}$}

\usepackage[T1]{fontenc}
\usepackage[english]{babel}                                
\usepackage[utf8]{inputenc}        
\usepackage{fourier}
\usepackage{amsfonts}
\usepackage{graphicx}
\usepackage{amsthm}
\usepackage{fancyhdr}
\usepackage[square,numbers]{natbib}
\usepackage{amsmath,amssymb,amscd,amsthm}
\usepackage{multicol,longtable,tabularx}
\usepackage{enumerate}
\usepackage{latexsym}
\usepackage{float}
\usepackage{amstext,graphicx,color}
\usepackage[english]{babel}
\usepackage{ifthen}
\usepackage{verbatim}
\usepackage{calc}
\usepackage{amsmath,amssymb,amsthm,mathrsfs}

\usepackage{multimedia,color,stmaryrd}

\allowdisplaybreaks

\newcommand{\R}{{\mathbb R}}

\newcommand{\EE}{{\mathcal {E}}}

\newtheorem{theorem}{Theorem}[section]

\newtheorem{lemma}[theorem]{Lemma}
\newtheorem{defn}[theorem]{Definition}

\newtheorem{rmk}[theorem]{Remark}

\begin{document}
\setcounter{page}{1}

\maketitle
\vspace{-0.8cm}
\begin{center} 
{ $^1$Institut Camille Jordan\\Universit\'e Claude Bernard Lyon 1\\ \vspace{0.1cm}}
\end{center}
\begin{center} 
{$^2$Center for Modelling and Simulation in the Biosciences - IWR\\University of Heidelberg }
\end{center}
\begin{center}
{\vspace{0.2cm}$^3$Istituto per le Applicazioni del Calcolo "Mauro Picone"\\
Consiglio Nazionale delle Ricerche \vspace{0.1cm}}\end{center}
\vspace{0.3cm}

\begin{abstract}
We consider a hyperbolic-parabolic model of vasculogenesis in the multidimensional case. For this system we show the global existence of smooth solutions to the Cauchy problem, using suitable energy estimates. Since this model does not enter in the classical framework of dissipative problems, we analyze it combining the features of the hyperbolic and the parabolic parts. Moreover we study the asymptotic behavior of those solutions showing their decay rates by means of detailed analysis of the Green function for the linearized problem.
\end{abstract}

\pagestyle{myheadings}
\thispagestyle{plain}
\markboth{C. DI RUSSO AND A. SEPE}{A HYPERBOLIC-PARABOLIC MODEL FOR VASCULOGENESIS}

\section{Introduction}
In this work we present some analytical results on the PDEs model of vasculogenesis proposed by Gamba el al. \cite{GaAmCo}.\\ 
Vasculogenesis is the process of blood vessel formation occurring by the production of endothelial cells and is lead by a chemotactic phenomenon, i.e. cells direct their movement according to certain chemicals in their environment \cite{Car,Hel}. 
At first, it was believed to occur only during embryologic development but recently it was realized that vasculogenesis can also occur in the adult organism. Circulating endothelial progenitor cells were identified and it was observed that they were able to contribute to neovascularization, such as during tumor growth, or to the revascularization process following a trauma, e.g., after cardiac ischemia.

In this paper we proceed in the mathematical study of the following system proposed by Gamba et al. \cite{GaAmCo}
\begin{equation}
\label{Gamba_0}
\left\{
\begin{array}{l}
\partial_t \rho +\nabla \cdot (\rho u)  = 0, \\\\
\partial_t (\rho u) + \nabla \cdot (\rho u \otimes u) +\nabla P(\rho)= -\alpha \rho u +\mu\rho \nabla \phi,\\\\
\partial_t \phi =D \Delta \phi +a \rho -b \phi.
\end{array}
\right.
\end{equation}
Here $\rho $ is the endothelial cells density, $u$ the cells velocity and $\phi$ the concentration of the chemoattractant.
Moreover, the positive constants $D$, $a$, and $b$ are, respectively, the diffusion coefficient, the rate of release and the inverse of the characteristic degradation
time of chemoattractant. The other two positive constants, $\alpha$ and $\mu$, measure respectively
the friction of the cells on the substrate and the strenght of the
cell response to the chemical signals. 

Classical arguments for hyperbolic-parabolic systems 
fail in the case of system (\ref{Gamba_0}) due to the production term in the
third equation, i.e.:  $a>0$. The treatment of these terms is not an easy task so, as a first step in the analytical study of this system, we focus on a particular case. 
Actually, we restrict ourselves to a specific class of solutions,
i.e.: to small perturbations of non null constant states. Indeed, we are able
to show the existance of global solutions for this Cauchy problem in
the full multidimensional setting, and some precise time decay
estimates for those solutions.
A first mathematical results on system (\ref{Gamba_0}) was proved by Kowalczyk et al. in \cite{KoGaPre}, where a viscous term $\gamma\nabla^2 v$ in the second equation is considered to introduce an energy mechanism that models the slowing down of cells in the proximity of network structure. They performed a detailed linear stability analysis of the model in the two dimensional case, with the aim of checking their potential for structure formation starting from initial data which represent a continuum cell monolayer. This model is unstable at low cell densities, while pressure stabilizes it at high densities.

In \cite{DiFDo}, Di Francesco and Donatelli dealt with diffusive relaxation limits of system \eqref{Gamba_0} toward
Keller-Segel type systems, either hyperbolic-parabolic or hyperbolic-elliptic. 
In order to produce a nontrivial class of solutions to the hyperbolic
system which after a proper rescaling relax toward a Keller-Segel type model,
they provided by means of Friedrich's symmetrization technique and by linearization arguments, an existence theorem (local in time) for the approximating system. Moreover they proved the uniform estimates needed to justify the assumptions in case of initial
densities which are small perturbation of an arbitrary non zero constant state.

More in general, hyperbolic-parabolic systems have been widely studied by Kawa-shima and Shizuta \cite{Ka,KaShi2,ShiKa}. Under the smallness assumption on the initial data and the dissipation condition on the linearized system, they were able to prove global (in time) existence and asymptotic stability of smooth solutions to the initial value problem for a general class of symmetric hyperbolic-parabolic systems.\\
However system (\ref{Gamba_0}) does not enter in this framework. As a matter of fact, due to the presence of the source terms $a \rho$, the dissipative condition fails.\\
If we linearize the differential part of (\ref{Gamba_0}) we get the semilinear hyperbolic-parabolic model, introduced by Hillen to describe chemosensitive movements   \cite{Hi2},
\begin{equation}
\label{sistema_GUMA}
\left\{
\begin{array}{l}
\partial_{t} \rho +\nabla \cdot v = 0, \\\\
\partial_{t} v + \gamma^2\nabla \rho= -c(\phi,\nabla\phi)v+h(\phi,\nabla \phi)g(\rho),\\\\
\partial_{t} \phi =\Delta \phi +f(\rho,\phi),
\end{array}
\right.
\end{equation}
where $\rho$ is the density of the population with finite speed $\gamma$, $v=\rho u$ the flux and $\phi$ the chemoattractant concentration. 
With reference to the one dimensional case, a first result of local and global existence for weak solutions, under the assumption of turning rate's boundness, was proved in \cite{HiSte}. Recently Guarguaglini et al. in \cite{GuMaNaRi} proved more general results of this model under weaker hypotheses, by showing a general result for global stability for a zero constant state in the Cauchy problem and for a small constant state in the
Neumann problem. These results have been obtained using the general theory of linearized operators, and an accurate analysis of their nonlinear perturbations.
Proceeding along these lines, in \cite{DiNaRi} the authors presented a global existence theorem and the asymptotic behavior for smooth solutions with small initial data to the Cauchy problem, for a simplified version of system (\ref{sistema_GUMA}) in the two dimensional case. Moreover in \cite{Di,Di2} it has been considered the multidimensional model (\ref{sistema_GUMA}), and showed the global existence of smooth solutions with small initial data to the Cauchy problem and determinated their asymptotic behavior.

Since in our case it is not possible to apply the technique of the semilinear case, we follow a different approach. 
The basic idea is to consider the hyperbolic and parabolic equation ``separately'', and to take advantage of their respective properties. Let us explain our approach in more details.

We look at the hyperbolic part of system (\ref{Gamba_0}) without the source term $\mu\rho\nabla \phi$, i.e. we consider isentropic Euler equations with damping. This system enters in the general framework proposed by Hanouzet and Natalini. In \cite{HaNa}, they determined  sufficient conditions which guarantee the global existence in time of smooth solutions for small initial
data, which are the entropy dissipation and the Shizuta-Kawashima conditions.\\
The first one is a condition for systems which are endowed with a strictly convex entropy. Even if the strict convexity guarantees that the entropy estimates are equivalent to the $L^2$ estimate, and the dissipation  yields the invariance in the same norm, this condition is too weak to prevent the formation of singularities. Indeed there exist systems that satisfy this condition such that there is no global solutions for some arbitrarily small initial data.
The condition (SK) is an adaptation to hyperbolic problems of the Kawashima condition for hyperbolic-parabolic ones. In terms of stability it guarantees the necessary coupling between conserved and non conserved quantities in order to have dissipation effects, in both the sets of state variables.\\ 
Following the work by Hanouzet and Natalini \cite{HaNa} and Yong \cite{Yo}, our approach is based on energy estimates for the parabolic and hyperbolic equations.
As a matter of fact, even if classical arguments fail in the estimate of the source term $\mu\rho\nabla \phi$, we are able to treat it thanks to particular estimates of the parabolic equation.

Once that the global existence for smooth solutions has been obtained for perturbation of small constant states, we are able to determine the asymptotic behavior for large times of solutions, by using the decay rates of the Green functions. Our strategy consists in using the decomposition of the Green function of dissipative hyperbolic systems done by Bianchini at al. \cite{BiHaNa} and its precise
decay rates. Indeed  in \cite{BiHaNa} the authors proposed a detailed description of the multidimensional Green function for a class of partially dissipative systems.
They analyzed the behavior of the Green function for the linearized problem, decomposing it into two main terms. The first
term is the diffusive one, and consists of heat kernels, while the faster term consists of the hyperbolic part. Moreover they gave a more precise description of the behavior of the diffusive part, which is decomposed into four blocks decaying with different rates, and the conservative one. 

By using these refined estimates we were able to determine the asymptotic behavior of smooth solutions. \\

The article is organized as follows. Next section deals with the modeling background relative to system (\ref{Gamba_0}).
In the second section we recall some basic results about dissipative hyperbolic systems satisfying the Shizuta-Kawashima condition. In the subsequent section we show the global existence of small solutions by means of energy estimates. Finally, the last section is devoted to the study of the decay properties of small and smooth solutions to the quasilinear system (\ref{Gamba_0}).

\section{Modeling Background}

The formation of new blood vessels, called vasculogenesis, is a process lead by a chemotactic phenomenon, i.e. cells direct their movement according to certain chemicals in their environment. As a matter of fact, recent works \cite{Car,Hel} have confirmed that endothelial cells in the process of vascular network formation exchange signals by the release and absorption of Vascular Endothelial Growth Factor
(VEGF-A). This growth factor can bind to specific receptors on the cell surface and induce
chemotactic motion along its concentration gradient \cite{Fer}. This communication by chemical signals determines how cells arrange and organize themselves. \\
As shown in \cite{Mu1,Mu2}, chemotaxis is decisive in many biological processes. For example, the formation of cells aggregations (amoebae, bacteria, etc.) occurs during the response of the populations to the change of the chemical concentrations in the environment. Moreover, also in multicellular organisms, chemotaxis of cells populations plays a crucial role throughout the life cycle: during embryonic development it is important in organizing
cell positioning, for example during gastrulation \cite{DorWe} and patterning of the nervous system \cite{PaWuRao}; in the adult life, it directs immune cells migration to sites of inflammation \cite{Wu}, fibroblasts into wounded regions and during cancer growth it allows tumor cells to invade the surrounding environment \cite{CoSiSe} or stimulating new blood vessel growth \cite{LaKa}.\\
This biological phenomenon can be described at different scales. For example, by considering the population density as a whole, it is possible to obtain macroscopic models of partial differential equations. One of the most celebrated model of this class is the parabolic one proposed by Patlak in 1953 \cite{Pa} and subsequently by Keller and Segel in 1970 \cite{KeSe3}.\\
However, the approach of PKS model is not always sufficiently precise to describe the biological phenomena \cite{FiLaPe}. As a matter of fact, diffusion leads to fast dissipation or explosive behaviors and prevents us to observe intermediate organized structures, like aggregations. This approach describe processes on a long time scale, while on a short time range one gets a suited description from models with finite characteristic speed.\\
Kinetic transport equations describe quite well the movement of single organism. For example the ``run and tumble'' (the movement along straight lines, the sudden stop and the change of direction) can be modeled by a stochastic process called  velocity-jump process \cite{HiOt,Ste}.

At an intermediate scale, the process can be described by means of hyperbolic differential equations. \\
This class of models can be derived as a fluid limit of transport equations, but with a different
scaling, namely the hydrodynamic scaling $t\rightarrow \epsilon t$, $x\rightarrow \epsilon x$ \cite{ChaMarPeSc}.

Hyperbolic models can also be obtained by phenomenological derivations, as done by Gamba et al. \cite{GaAmCo,SeAmGi} to describe the vasculogenesis process. In \cite{GaAmCo} the authors proposed a model including chemotaxis  as a fundamental mechanisms for cell-to-cell communication in order to find key parameters in the complexity of the formation of vascular network. This biological process proceeds along three main stages: migration and early network formation, network remodeling and differentiation in tubular structures. The model proposed in \cite{GaAmCo} focused on the first stage of the process. 

The experimental results are encoded into a mathematical model starting from the assumptions that
the cell population can be described by
a continuous distribution of density $\rho$ and velocity $u$, moreover it is 
also assumed the presence of a concentration $\phi$ of chemoattractant.
The cell population in the early stages of its
evolution can be modeled as a fluid
of non-directly interacting particles and is accelerated
by gradient of chemoattractant released by cells, which diffuse and degrade in
finite time.

From these assumptions follows the system
\begin{equation}
\label{Gamba_I}
\left\{
\begin{array}{l}
\partial_t \rho +\nabla \cdot (\rho u)  = 0, \\\\
\partial_t (\rho u) + \nabla \cdot (\rho u \otimes u) +\nabla P(\rho)= -\alpha \rho u +\mu\rho \nabla \phi,\\\\
\partial_t \phi =D \Delta \phi +a \rho -b \phi,
\end{array}
\right.
\end{equation}
where $D$, $a$, and $b$ are, respectively, the diffusion coefficient, the rate of release and the inverse of the characteristic degradation
time of chemoattractant, $\mu$ measures the strength of
cell response, and $\alpha$ the friction of the cells on the substrate. 

This system is derived in a classical way by Continuum Mechanics indeed the first equation describes mass conservation, the second one is a momentum balance with a chemotactic
force and the last is a diffusion equation for the
chemoattractant produced by endothelial cells and degrading in time;
in particular, the convective term on the right hand side of the second equation allows describing cell migration. 
As indicated in \cite{AmBuPre} it is based on the assumptions that: (i) endothelial cells show persistence in their motion;
(ii) endothelial cells communicate via the release and absorption of molecules of a soluble growth
factor and this chemical factor can be reasonably identified with VEGF-A (Serini et al. \cite{SeAmGi}); (iii) the chemical factors, released by cells, diffuse and degrade in time; (iv) endothelial cells neither duplicate nor die during the process;
(v) cells are slowed down by friction due to the interaction with the fixed substratum;
(vi) closely packed cells mechanically respond to avoid overcrowding.

On the basis of experiments and theoretical insights, the authors in \cite{SeAmGi} showed that non-linear mechanics and chemotactic cellular dynamics fit into a model able to reproduce with great accuracy the formation of capillary networks in vitro.

The model (\ref{Gamba_I}) is able to reproduce several experimentally observed facts, e.g. that the mean chord length is approximately independent on the initial cell density or that connected networks are formed only above a critical initial density as shown in \cite{GaAmCo,SeAmGi}. Moreover the authors provided a strong evidence that endothelial cells number and the range of activity of a chemoattractant factor regulate vascular network formation by flanking biological experiment, theoretical insights, and numerical simulations.

\section{System Properties}

Let us consider the Cauchy problem for the following hyperbolic-parabolic system
\begin{equation}
\label{sistemaPer}
\left\{
\begin{array}{l}
\partial_t \tilde{\rho} +\nabla \cdot \tilde{v}  = 0, \\\\
\partial_t  \tilde{v} + \nabla \cdot \left(\tilde{v}\otimes \frac{\tilde{v}}{\tilde{\rho}}\right) +\nabla P(\tilde{\rho})= -\alpha\tilde{v} +\mu\tilde{\rho} \nabla \tilde{\phi},\\\\
\partial_t \tilde{\phi} =D \Delta\tilde{\phi} +a \tilde{\rho} -b{\tilde{\phi}},
\end{array}
\right.	
\end{equation}
\begin{equation}\label{Cauchy}
\tilde{\rho}(x,0)=\rho_0(x),\qquad \tilde{u}(x,0)=u_0(x), \qquad \tilde{\phi}(x,0)=\phi_0(x).
\end{equation}
We made the assumption
$$
P^\prime(\tilde{\rho})>0
$$ 
to ensure the strictly hyperbolicity of system (\ref{sistemaPer}). \\
Our aim is to prove that, under suitable assumptions, the Cauchy problem (\ref{sistemaPer})-(\ref{Cauchy})
admits a global smooth solution for small initial data. In particular, we are interested in solutions of the form $(\tilde{\rho},\tilde{v},\tilde{\phi})= (\rho+\overline{\rho},v,\phi+\overline{\phi})$, where $(\overline{\rho}, 0 ,\overline{\phi})$ is a constant stationary solution to the problem with $\overline{\rho}>0$ and $\overline{\phi}=\frac{a}{b}\overline{\rho}$, and $(\rho,v,\phi)$ is a  perturbation. In this case, we can rewrite system (\ref{sistemaPer}) as follows:
\begin{equation}
\label{sistemaPper}
\left\{
\begin{array}{l}
\partial_t \rho +\nabla \cdot v  = 0, \\\\
\partial_t v + \nabla \cdot \left(v\otimes \frac{{v}}{\rho+\overline{\rho}}\right) +\nabla P(\rho+\overline{\rho})= -\alpha v +\mu(\rho +\overline{\rho})\nabla \phi,\\\\
\partial_t \phi =D \Delta \phi +a \rho -b{\phi}. 
\end{array}
\right.
\end{equation}

Now, we show some properties of the hyperbolic-parabolic system (\ref{sistemaPper}). To this end, we rewrite it in the following compact form:
\begin{equation}
\left\{
\begin{array}{l}
\partial_{t} U +  \sum\limits_{j=1}^n \partial_{x_j}f_j (U+\overline{U})=g(U+\overline{U})  + h(U+\overline{U},\nabla\phi),\\\\
\partial_t {\phi} =D \Delta {\phi} +a { \rho} -b  {\phi},
\end{array}
\right.\nonumber
\end{equation}
where $U=({\rho},{v})$, $\overline{U}=(\overline{\rho},0)$, $f_j(U+\overline{U})=\left({v}_j, \frac{{v}_1 {v}_j}{\rho+\overline{\rho}},\ldots, \frac{{v}_j^2}{\rho+\overline{\rho}}+ P(\rho+\overline{\rho}),\ldots,\frac{{v}_n {v}_j}{\rho+\overline{\rho}} \right)$, $g(U+\overline{U})=(0, -\alpha {v})$ and $h(U+\overline{U},\nabla {\phi})=(0,\mu (\rho+\overline{\rho}) \nabla{\phi})$.


\subsection{Strictly Entropy Dissipative Condition}\label{LocStrDiss}
Let us consider the hyperbolic part of  (\ref{sistemaPper}), that is 
\begin{equation}
\label{sistHyp}
\left\{
\begin{array}{l}
\partial_t \rho +\nabla \cdot v = 0, \\\\
\partial_t v + \nabla \cdot \left(v\otimes \frac{{v}}{\rho+\overline{\rho}}\right) +\nabla P(\rho+\overline{\rho})= -\alpha v.
\end{array}
\right.
\end{equation}

First of all, we want to prove that system (\ref{sistHyp}) is endowed with an entropy function, that is a convex real function $\mathcal{E}$ such that there exist related entropy-fluxes $q_j$ satisfying the following condition
\begin{align}
	(f_j')^t\nabla \mathcal{E} = q_j', \nonumber
\end{align}
where $f_j$ are the fluxes of the system (\ref{sistHyp}), for $j=1,...,n$. \\
To ensure the existence of  entropy-fluxes $q_j$, it is sufficient to prove that 
\begin{align}
( f_j')^t \mathcal{E}'' \quad \mbox{is symmetric},\nonumber
\end{align}
or equivalently that $\mathcal{E}'' f_j'$ is symmetric, for each $j=1,...,n$ (see \cite{Bou1}). 

Once we have proved the existence of an entropy function for system (\ref{sistHyp}), the following additional equation for the entropy evolution can be written:
\begin{align}
	\partial_t (\mathcal{E}(U)-\nabla \mathcal{E}(\hat{U}) \cdot U)&+  \sum_{j=1}^n \partial_{x_j} (q_j(U)-\nabla \mathcal{E}(\hat{U}) f_j(U))\nonumber \\
	&= (\nabla \mathcal{E}(U)-\nabla \mathcal{E}(\hat{U}) )\cdot(g(U)-g(\hat{U})), \nonumber
\end{align}
where $\hat{U}$ is an equilibrium state for the system (\ref{sistHyp}) (i.e. $g(\hat{U})=0$). From this equation, we deduce that the integral of $\mathcal{E}(U) - \nabla \mathcal{E}(\hat{U})\cdot U$ is decreasing in time, if the term on the right-hand side is negative.\\
Denoted by $\gamma$ the set of equilibrium states of the system (\ref{sistHyp}), this property is encoded in the following definition.

\begin{defn}
An entropy $\mathcal{E}$ for the system (\ref{sistHyp}) is dissipative at $\hat{U}$, with $\hat{U} \in \gamma$, if it satisfies the inequality
\begin{align}
	(\nabla \mathcal{E}(U) - \nabla \mathcal{E}(\hat{U})) \cdot (g(U)-g(\hat{U})) \le 0,\nonumber 
\end{align}
for any $U$ in a neighborhood of $\hat{U}$.
\end{defn}

Clearly, an entropy function is dissipative, if it is dissipative at $\hat{U}$, for each $\hat{U} \in \gamma$.

Next, if $\mathcal{{E}}$ is a strictly convex function, we can introduce the entropy variable
\begin{align}
	W:= \nabla \mathcal{E}(U), \nonumber
\end{align}
and the functions
\begin{align}
	\mathcal{{E}}^*(W)&:= W \cdot \Phi (W)- \mathcal{{E}}(\Phi(W)), \nonumber\\
	q_j^*(W)&:= W \cdot f_j(\Phi(W))- q_j(\Phi(W)),\nonumber
\end{align}
where $\Phi=(\Phi_1,\Phi_2):= (\nabla \mathcal{E})^{-1}$. 
Let us set now $A_0=(\mathcal{{E}}^*)''(W)$, $A_j=f_j'(\Phi(W))A_0$ and ${G}(W)=g(\Phi(W))$, then we rewrite the system  (\ref{sistHyp}) in the entropy variable as
\begin{align}
A_0 \partial_t W + \sum_{j=1}^n A_j \partial_{x_j} W = {G} (W). \label{sistHyp_W}
\end{align}
Let us observe that the matrix $A_0$ is symmetric positive definite and $A_j$ is symmetric, for each $j=1,...,n$. 

Now, we take $\mathcal{U}$ an open subset of $\R^{n+1}$ and define
\begin{align}
\gamma&:= \left\{U \in \mathcal{U} \, : \,\, g(U)=0 \right\}, \nonumber\\
	\Gamma&:=\nabla \mathcal{{E}}(\gamma)=\left\{W \in \nabla \mathcal{E}(\mathcal{U}) \, : \,\, G(W)=0 \right\}. \nonumber
\end{align}
Let us observe that  the dissipative condition ensures the existence of a real positive matrix $B=B(W, \overline{W})$ such that, for every $W$ in a suitable neighborhood of $\overline{W}$,
\begin{align}
	G_2(W)= - B(W, \overline{W}) (W_2-\overline{W}_2). \nonumber 
\end{align}
We refer to \cite{HaNa} for more details. 
\begin{defn}
The system (\ref{sistHyp}) is strictly entropy dissipative, if there exists a real matrix  $B(W, \overline{W})\in \R^{n}\times\R^{n} $, positive definite, such that
\begin{align}
	G_2(W)= - B(W, \overline{W}) (W_2-\overline{W}_2), \nonumber 
\end{align}
for every $W\in \nabla \mathcal{E}(\mathcal{U})$ and $\overline{W}=(\overline{W_1},\overline{W_2})\in \Gamma$.
\end{defn}

In our case, we can consider for system (\ref{sistHyp}) the canonical entropy function
\begin{equation}
\mathcal{{E}}(\rho+\overline{\rho},v)=\frac{1}{2}\frac{v^2}{\rho+\overline{\rho}}+(\rho+\overline{\rho})\int_0^{\rho+\overline{\rho}}\frac{P(\tau)}{\tau^2}d \tau,\nonumber
\end{equation}
then 
\begin{align*}
\mathcal{{E}}_\rho(\rho+\overline{\rho},v)&=-\frac{1}{2}\frac{v^2}{(\rho+\overline{\rho})^2}+\frac{P(\rho+\overline{\rho})}{\rho+\overline{\rho}}+\int_0^{\rho+\overline{\rho}}\frac{P(\tau)}{\tau^2}d \tau,\\
\mathcal{E}_v(\rho+\overline{\rho},v)&=\frac{v}{\rho+\overline{\rho}}.
\end{align*}
It is easy to prove that system (\ref{sistHyp}), endowed with the entropy $\mathcal{E}$, satisfies the \textit{strictly entropy dissipative condition}.
Moreover, according to \cite{HaNa}, the definition of dissipative entropy is invariant for affine perturbation, so if we consider the function 
\begin{align}
	\mathcal{\tilde{E}}(U)= \mathcal{E}(U+\overline{U})-\mathcal{E}(\overline{U})- \nabla \mathcal{E}(\overline{U}) \cdot U, \label{entropia}
\end{align}
we have that $\mathcal{\tilde{E}}(U)$ is a quadratic function and still an entropy for (\ref{sistHyp}). This system, endowed with $\mathcal{\tilde{E}}(U)$, satisfies the strictly entropy dissipative condition as well. \\     
 Let us finally observe that, in our case, 
\begin{align}
	W:=\nabla  \mathcal{\tilde{E}}(U)=\nabla \mathcal{E}(U+ \overline{U})-\nabla \mathcal{E}(\overline{U}), \nonumber
\end{align}
and
\begin{align}
	\gamma= 
	\left\{U \in \mathcal{U} \, : \,\, U=(\rho,0) \right\},\nonumber
\end{align}
\begin{align}
	\Gamma= 
	\left\{W \in\nabla  \mathcal{\tilde{E}}(\mathcal{U}) \, : \,\, W_2=0 \right\}. \nonumber
\end{align}


\subsection{The Shizuta-Kawashima Condition}\label{SKcond}
This section is devoted to prove that system (\ref{sistHyp_W}) satisfies the Shizuta-Kawashima condition \cite{ShiKa} that is: 
\begin{description}
	\item[(SK)]  every eigenvector of $\sum\limits_{j=1}^{n} f'_j(\overline{U})\xi_j$ is not in the null space of $g'(\overline{U})$, for every $\xi\in \R^n -\{0\}$.
\end{description}

Let us define $A(U+\overline{U})=\sum\limits_{j=1}^n f_j'(U+\overline{U}){\xi_j}$, with $\xi\in \R^n-\{0\}$, so 
\begin{align}
	A(\overline{U})=\left(
\begin{array}{cccccc}
0 &\xi_1& \ldots  & \xi_n \\
 P'(\overline{\rho})\xi_1 & 0&\ldots  &0\\
\ldots &    \ldots &    \ldots &     \ldots \\
P'(\overline{\rho})\xi_j& 0 &\ldots &0\\
\ldots &    \ldots &       \ldots &    \ldots \\
P'(\overline{\rho})\xi_n& 0 &\ldots &0
\end{array}
\right).\nonumber
\end{align}
Now, we suppose that $X$ is in the null space of $g'(\overline{U})$, that is 
\begin{align}
	g'(\overline{U})X=0  \qquad &\Longleftrightarrow \qquad
X_j=0 \quad  \textrm{ for }\,\, j=2,\ldots, n+1. \nonumber
\end{align}
So if $X$ is in the null space of $g'(\overline{U})$, then $X=(X_1,0,\ldots,0)^t$. Since in this case we have 
\begin{align}
	\lambda X= A(\overline{U}) X  \qquad  \Longleftrightarrow \qquad \lambda X_1=0, \nonumber
\end{align}
$X$ cannot be an eigenvector of $A(\overline{U})$. But this is absurd because $X$ is an eigenvector of $A(\overline{U})$, therefore $X$ is not in the null space of $g'(\overline{U})$. This proves that system (\ref{sistHyp_W}) satisfies the Shizuta-Kawashima condition.


\section{The Global Existence of Smooth Solution}

In this section, by means of the entropy method, we aim to prove the global existence of smooth solution to the complete hyperbolic-parabolic system
\begin{equation}
\label{generale_perturbato1}
\left\{
\begin{array}{l}
\partial_t \rho +\nabla \cdot v  = 0, \\\\
\partial_t v + \nabla \cdot \left(\frac{v^2}{\rho+\overline{\rho}}+P(\rho+\overline{\rho})\right)= -\alpha v +\mu(\rho +\overline{\rho})\nabla \phi,\\\\
\partial_t \phi =D \Delta \phi +a \rho -b{\phi}.
\end{array}
\right.
\end{equation}
Let us recall that ${\rho}, {\phi}:\R^n\times \R^+ \rightarrow \R^+$, $u:\R^n\times \R^+ \rightarrow \R^n$, $v:= (\rho+ \overline{\rho}) u$, and $P^\prime({\rho}+\overline{\rho})>0$. Moreover, $\overline{U}=(\overline{\rho},0, \overline{\phi})$ is a constant stationary solution to the problem, with $\overline{\phi}=\frac{a}{b}\overline{\rho}$.\\
As we have shown in previous sections, the hyperbolic part (\ref{sistHyp}) of  system (\ref{generale_perturbato1}) is endowed with the dissipative entropy (\ref{entropia}) and it satisfies the strictly entropy dissipative condition.  So, considering the entropy variable $W=\nabla \mathcal{\tilde{E}}(U)$ and, setting $\Phi(W)=(\nabla \mathcal{\tilde{E}})^{-1}(W)$ and  
\begin{align}
	A_0(W)&=(\Phi(W))^\prime, \qquad \qquad \quad A_j(W)=f_j^\prime(\Phi(W))A_0,\nonumber \\
	{G}(W)&=g(\Phi(W)), \quad \quad \quad \,{H}(W,\nabla \phi)=h(\Phi(W),\nabla \phi),\nonumber
\end{align}
we write our system as:
\begin{equation}
\label{variabili_entro}
\left\{
\begin{array}{l}
	A_0\partial_t W+\sum\limits_{j=1}^n  A_j\partial_{x_j} W={G}(W)+ {H}(W, \nabla \phi),\\\\
\partial_t \phi =D \Delta \phi +a \Phi_1(W) -b\phi.
\end{array}
\right.
\end{equation}
Let us notice that the existence of a local solution to system (\ref{variabili_entro}) is ensured by classical argument. Indeed, (\ref{variabili_entro})  is a symmetric hyperbolic-parabolic system, therefore
we know that, if initial data $W_0,\,\phi_0$ are in $H^s(\R^n)$, with $s> [n/2]+1$, then there exists a local in time solution $(W,\phi)\in C([0,T),H^s(\R^n))\times (C([0,T),H^s(\R^n))\cap L^2([0,T),H^{s+1}(\R^n)))$ for system (\ref{variabili_entro}) (Theorem 2.9, \cite{Ka}). \\ 
Then, thanks to the continuation principle (see \cite{Ma}), in order to prove that $(W,\phi)$ is global in time, it is sufficient to show uniformly (in time) estimates of the local solution. \\
We state now our result: 
\begin{theorem}\label{global}
Fix $s > [n/2]+1=:s_0$. We consider the Cauchy problem associated to system (\ref{variabili_entro}), with small initial data $W_0$ and $\phi_0$ in $H^{s}(\R^n)$. If $\left\|W_0\right\|_{H^s}$, $\left\| \phi_0\right\|_{H^{s}}$ and $\overline{\rho}$ are sufficiently small, then there exists a unique solution $(W,\phi)$ of system (\ref{variabili_entro}), such that
\begin{eqnarray}
	W \in C([0,\infty),H^s(\R^n)), \qquad 	\phi \in C([0,\infty),H^{s}(\R^n))\cap L^2([0,\infty),H^{s+1}(\R^n)), \nonumber
\end{eqnarray}
and for each $t>0$
\begin{eqnarray}
&&	\left\|W(t)\right\|_{H^s}^2+\int_0^t\left\|W_2(\tau)\right\|_{H^s}^2 d \tau +\int_0^t \left\|\nabla W(\tau)\right\|^2_{H^{s-1}}d\tau \le C \left\|W_0 \right\|_{H^s}^2,
	 \nonumber\\ \nonumber\\
&&	 \left\|\phi(t)\right\|_{H^s}^2+\int_0^t \left\|\nabla \phi(\tau)\right\|^2_{H^s}d\tau \le C\left\|\phi_0 \right\|_{H^s}^2, \nonumber
\end{eqnarray}
where $C=C(\left\|W_0\right\|_{H^s}, \left\|\phi_0\right\|_{H^{s}},\overline{\rho})$. 
\end{theorem}
In order to prove this theorem, we firstly show the validity of some energy estimates for functions $W$ and $\phi$.
\begin{rmk}
This result holds for perturbation of small constant (non null) states. Moreover in the isothermal case, i.e. $P(\rho)=\rho$, it holds also for perturbation of zero state in the one dimensional case. It can be proved using the result of \cite{MaNa} that ensures the existence of a dissipative entropy for $2 \times 2 $ hyperbolic system, under suitable assumptions \cite{DiSe2}. \\
\end{rmk}
\begin{rmk}
	Concerning the constants, they all have been denoted by the letter $c$. Thus, $c$ may stand for numbers that are different from line to line of the text. Only when we intend to explicitly indicate the dependence of $c$ on some parameters, or to avoid confusions, we have used some other notations for the constants.
\end{rmk}


\subsection{Energy Estimates for $\phi$}\label{Est_phi}
We consider the local (in time) solution $(W,\phi)\in C([0,t),H^s(\R^n))\times (C([0,t),H^s(\R^n))\cap L^2([0,t),H^{s+1}(\R^n)))$, where $t \in [0,T)$.\\
First of all, we consider the parabolic equation
\begin{align}
	\partial_t \phi =D \Delta \phi +a \rho -b{\phi}. \label{parEq}
\end{align}
Applying the spatial derivative of order $\alpha$, with $\alpha=(\alpha_1,...,\alpha_n)$ and $0 \le | \alpha | \le s$, 
and multiplying by $\partial^\alpha_x \phi$, we get 
\begin{align}
	\partial_t \left(\frac{1}{2}(\partial^\alpha_x \phi)^2\right) &+D \sum_{j=1}^n \left(\partial_{x_1}^{\alpha_1} \partial_{x_2}^{\alpha_2}... \partial_{x_j}^{\alpha_j+1}...\partial^{\alpha_n}_{x_n} \phi\right)^2 \nonumber \\
	&=D \sum_{j=1}^n \partial_{x_j} \left(\partial_{x_1}^{\alpha_1} \partial_{x_2}^{\alpha_2}... \partial_{x_j}^{\alpha_j+1}...\partial^{\alpha_n}_{x_n} \phi \, \partial_x^\alpha \phi\right) +a \partial_x^\alpha \rho \partial^\alpha_x \phi -b{(\partial_x^\alpha \phi)^2}. \nonumber 
\end{align}
Then, integrating with respect to $x$ and $t$, we have
\begin{align}
\frac{1}{2}	\int 	(\partial^\alpha_x \phi)^2 dx +&  D \int_0^t \int \left( \sum_{j=1}^n \left(\partial_{x_1}^{\alpha_1} \partial_{x_2}^{\alpha_2}... \partial_{x_j}^{\alpha_j+1}...\partial^{\alpha_n}_{x_n} \phi\right)^2 \right)  d\tau dx  \nonumber \\
\le &\frac{1}{2}\int 	(\partial^\alpha_x \phi_0)^2 dx + \frac{a}{2\epsilon}\int _0^t\int  ( \partial^\alpha_x \rho)^2 dx d\tau \nonumber \\
&	+\left( \frac{a \epsilon }{2}-b\right)\int _0^t\int (\partial^\alpha_x \phi)^2 dx d\tau. \nonumber
\end{align}
Now, let us introduce the generic functional
\begin{align}
\qquad\quad\qquad	N_l^2(t):= \sup_{0 \le \tau \le t} \left\|W(\tau)\right\|_{H^l}^2 + \int_0^t \left\|W_2(\tau)\right\|_{H^l}^2 d\tau + \int_0^t \left\|\nabla W(\tau)\right\|^2_{H^{l-1}} d \tau, \nonumber
\end{align}
for $l=1,...,s$, and
\begin{align}
	N_0^2(t):= \sup_{0 \le \tau \le t} \left\|W(\tau)\right\|_{L^2}^2 + \int_0^t \left\|W_2(\tau)\right\|_{L^2}^2 d\tau. \nonumber
\end{align}
Therefore, summing up the estimate of $\partial^\alpha_x \phi$ for $\alpha$ such that $|\alpha| \in [1,s]$ and assuming  $\epsilon$ sufficiently small, we obtain the $s$-order estimate for function $\phi$:
\begin{align}
\left\| \phi \right\|_{H^s}^2 + c \int_0^t \left\| \nabla \phi\right\|_{H^{s}}^2 d\tau  + c \int_0^t \left\| \nabla \phi\right\|_{H^{s-1}}^2 d\tau  
	\le \left\|\phi_0\right\|_{H^s}^2 +c N_s^2(t).\label{phix}
\end{align}
\begin{rmk}\label{stima_phix}
	Let us notice that the previous inequality allows us to estimate the integral in time of $\,\,\left\| \nabla \phi\right\|_{H^{s}}$ simply using the functional $N_s(t)$, which involves the derivatives up to the order $s$ of $W$.
\end{rmk}

Now, the function $\phi$ is a solution of the parabolic equation (\ref{parEq}), therefore, using the Duhamel's formula, it could be written as
\begin{eqnarray}
	\phi(x,t)= e^{-bt} \Gamma^p(x,t) * \phi_0(x) + \int_0^t e^{-b(t-\tau)} \Gamma^p(x,t- \tau) * \rho (x,t) d\tau,\nonumber
\end{eqnarray}
where $\Gamma^p$ is the heat kernel. Consequently, we obtain 
\begin{align}
\left\|\nabla \phi(x,t) \right\|_{L^\infty} 
 &\le \sum_{j=1}^n \left[ c e^{-bt} \left\|\partial_{x_j} \phi_0\right\|_{L^\infty} +  \sup_{0 \le \tau \le t} \left\|\rho(\tau)\right\|_{L^\infty} \int_0^t e^{-b(t-\tau)} (t-\tau)^{-1/2}d \tau\right] \nonumber \\
 &\le c e^{-bt} \left(\left\|\phi_0\right\|_{H^{s+1}} +  N_s(t)\right).\label{nabla_phi_Linfty}
\end{align}


\subsection{The zero-order Energy Estimate}

Now, we want to estimate the $L^2$-norm of the function $W$. To this end, let us rewrite the first two equations of system (\ref{generale_perturbato1}) in the form:
\begin{equation}\label{generale}
\partial_t U +\sum_{j=1}^n \partial_{x_j} f_j(U+ \overline{U})=g(U+ \overline{U})+h(U+ \overline{U},\nabla \phi).
\end{equation}
Multiplying the previous system by $\nabla  \mathcal{\tilde{E}}(U)= \nabla \mathcal{E}(U+ \overline{U})-\nabla \mathcal{E}(\overline{U}) $, we have
\begin{align}
	\partial_t \mathcal{\tilde{E}}(U)+ \sum_{j=1}^n \partial_{x_j} \tilde{q}_j(U)=\nabla \mathcal{\tilde{E}}(U)\cdot g(U+\overline{U}) +\nabla \tilde{\EE}(U)\cdot h(U+\overline{U},\nabla \phi), \nonumber
\end{align}
where $\tilde{q}_j$ ($j=1,...,n$) are the entropy-fluxes associated to the function $\mathcal{\tilde{E}}$.\\
Let us observe that, thanks to definitions of the entropy $\mathcal{\tilde{E}}$ and the variable $W=\nabla  \mathcal{\tilde{E}}(U)$, there exist two constants $\delta_0$ and $c$ such that 
$$
\frac{1}{c}|W|^2\leq \tilde{\EE}(U)\leq c|W|^2,
$$
for $|W| \le \delta_0$. Moreover, as proved in Section \ref{LocStrDiss}, the system (\ref{generale}) satisfies the strictly dissipative condition, therefore there exists a constant $c$ such that
$$
-(W\cdot {G}(W))\geq c |W_2|^2.
$$
Let us integrate the previous system, with respect to space variable $x$, so we get:
$$
\frac{d}{dt}\int \tilde{{\EE}}(U)dx =\int\nabla \tilde{\EE}(U)\cdot g(U+\overline{U})dx+\int\nabla \tilde{\EE}(U)\cdot h(U+\overline{U}, \nabla \phi )dx,
$$
which yields
$$
\|W(t)\|^2_{L^2}+c\int_0^t \|W_2(\tau)\|_{L^2}^2d\tau \leq \|W_0\|^2_{L^2}+c \int_0^t\int \nabla \EE(U+\overline{U})\cdot h(U+\overline{U}, \nabla \phi) dxd\tau, 
$$
for all $|W(x,t)| \le \delta_0$ where $(x,t) \in \R^n \times(0,T)$.\\
Now, let us observe that, thanks to the definition of function $h(U+\overline{U}, \nabla \phi)$, the last integral can be estimate as follows
\begin{align*}
\int_0^t\int\nabla \tilde{\EE}(U) \cdot h(U+\overline{U}, \nabla \phi) dxd\tau =&\int_0^t\int W_2 \cdot \mu( \rho+\overline{\rho}) \nabla \phi dxd\tau\nonumber \\
\leq& \int_0^t \mu \|W_2( \tau)\|_{L^2}\|\rho( \tau)\|_{L^\infty}\|\nabla \phi( \tau)\|_{L^2} d\tau \nonumber \\
&+ \int_0^t \mu\overline{\rho} \|W_2( \tau)\|_{L^2}\|\nabla \phi( \tau)\|_{L^2} d\tau \\
\leq& c \sup_{ \tau \in (0,t)} \|\rho( \tau)\|_{L^\infty}  \int_0^t \left(\|W_2( \tau)\|_{L^2}^2 + \|\nabla \phi( \tau)\|^2_{L^2} \right)d\tau\\
&+ c \overline{\rho}\int_0^t\left(\|W_2( \tau)\|^2_{L^2}+ \|\nabla \phi( \tau)\|_{L^2}^2 \right) d\tau\\
\leq&  c N_1(t) \left[ \left\| \phi_0\right\|_{L^2}^2 + N_1^2(t) \right] + c \overline{\rho}\left[\left\| \phi_0 \right\|_{L^2}^2 + N_1^2(t) \right],
\end{align*}
where, in the last inequality, we used the energy estimate of the function $\phi$. \\
In conclusion, the zero order estimate of function $W$ is given by 
\begin{align}
\|W(t)\|^2_{L^2}+\int_0^t\|W_2(\tau)\|^2_{L^2}d\tau \leq& N_0^2(0)+ C(\left\| \phi_0\right\|_{L^2}) N_1(t) \nonumber \\
&+ C(\overline{\rho}) N_1^2(t) +C N_1^3(t)+ C(\left\|\phi_0 \right\|_{L^2},\overline{\rho}). \nonumber
\end{align}


\subsection{The $s$-order Energy Estimate for function $W$}

To prove energy estimates of the function $\partial_x^s W$, it is necessary to use some inequalities based on the Sobolev embedding theorem. Here we just state the following lemma, whose proof can be found in \cite{Yo}.
\begin{lemma}\label{cal}
We take $s, \, s_1$ and $s_2$ three non-negative integers and $s_0:= [n/2]+1$. Then
\begin{description}
	\item[(i)] if $s_3 = \min \left\{s_1, s_2, s_1 + s_2 - s_0\right\} \ge  0$, then $H^{s_1}H^{s_2} \subset H^{s_3} $\\
	          (the inclusion symbol $\subset$ denotes the continuous embedding);
	\item[(ii)] if $s > s_0 $ and  $A',\, U \in H^{s-1}$, then for all multi-indices $\alpha$ with $0 \le | \alpha | \le s$, the commutator $[\partial_x^\alpha, A]U:= \partial_x^\alpha(AU)-A \partial_x^\alpha U \, \in L^2$ and
\begin{align}
	\left\|[\partial_x^\alpha, A]U\right\|_{L^2} \le C_s \left\|A'\right\|_{H^{s-1}} \left\|U\right\|_{H^{|\alpha|-1}}; \nonumber
\end{align}
	\item[(iii)] if $s \ge s_0$, $V \in H^s$ with values in $\Omega$, and $A \in {C}^s (\Omega)$ with $A(0) = 0$, then $A(V (\cdot))\in H^s$ and
\begin{align}
	\left\|A(V(\cdot))\right\|_{H^s} \le C_s |A|_{s} \left\|V\right\|_{H^s} (1+\left\|V\right\|_{H^s}^{s-1}). \nonumber
\end{align}
\end{description}
Here, $C_s$ is a constant depending only on $s$ and $n$, and $$|A|_s :=\sup\limits_{U \in \Omega, 1 \le | \alpha | \le s} |\partial_U^\alpha A(U)|.$$
\end{lemma}

Now we estimate the $L^2$-norm of the $s$-order derivative of the local function $W$. To this end, we consider the system  
\begin{equation}\label{A0W}
\partial_t W+\sum^{n}_{j=1} \widetilde{A}_j \partial_{x_j}  W=A_0^{-1}{G}(W)+A_0^{-1}{H}(W, \nabla \phi), 
\end{equation}
where $\widetilde{A}_j := A_0^{-1}A_j$. Applying the derivative $\partial^\alpha_x$, where $1 \le |\alpha| \le s$, we get
\begin{eqnarray}
 \partial_x^\alpha \partial_t W+ \sum^n_{j=1} \widetilde{A}_j \partial_x^\alpha \partial_{x_j} W =& A_0^{-1} \partial_x^\alpha G + [\partial_x^\alpha, A_0^{-1}]G + A_0^{-1} \partial_x^\alpha H \nonumber \\
 &+ [\partial_x^\alpha, A_0^{-1}]H+ \sum^n_{j=1} [\widetilde{A}_j, \partial_x^\alpha]  \partial_{x_j} W,\nonumber
\end{eqnarray}
where $[a,b]c:=a(bc)-b(ac)$.\\
If we multiply this equation by $(\partial_x^\alpha W)^t A_0$, we have
\begin{align}
 (\partial_x^\alpha W)^t A_0 \partial_x^\alpha \partial_t W&+ \sum^n_{j=1} (\partial_x^\alpha W)^t A_j \partial_x^\alpha \partial_{x_j} W = (\partial_x^\alpha W)^t \partial_x^\alpha G + (\partial_x^\alpha W)^t A_0[\partial_x^\alpha, A_0^{-1}]G \nonumber \\
 +(\partial_x^\alpha W)^t \partial_x^\alpha H &+ (\partial_x^\alpha W)^t A_0[\partial_x^\alpha, A_0^{-1}]H+ \sum^n_{j=1} (\partial_x^\alpha W)^t A_0 [\widetilde{A}_j, \partial_x^\alpha]  \partial_{x_j} W.\label{m1.6}
\end{align}
Thanks to the symmetry of $A_0$ and $A_j$, we deduce the following equalities
\begin{align}
  (\partial_x^\alpha W)^t A_0 \partial_x^\alpha \partial_t W=& \frac{1}{2} \partial_t \left((\partial_x^\alpha W)^t A_0 \partial_x^\alpha W	\right)- \frac{1}{2} (\partial_x^\alpha W)^t \left(\partial_t A_0 \partial_x^\alpha W\right), \nonumber\\
  \sum^n_{j=1} (\partial_x^\alpha W)^t A_j  \partial_x^\alpha \partial_{x_j} W=& \frac{1}{2} \sum^n_{j=1} \partial_{x_j} \left((\partial_x^\alpha W)^t A_j \partial_x^\alpha W\right)-\frac{1}{2} \sum^n_{j=1} (\partial_x^\alpha W)^t \partial_{x_j} A_j \partial_x^\alpha W.\nonumber 
\end{align}

Let us observe that, thanks to the strictly dissipative condition, there exists a positive definite matrix $B$ such that 
\begin{align}
	(\partial_x^\alpha W)^t \partial_x^\alpha G=
	- ( \partial_x^\alpha W_2)^t \partial_x^\alpha (BW_2)= -( \partial_x^\alpha W_2)^t B ( \partial_x^\alpha W_2) + ( \partial_x^\alpha W_2)^t [B, \partial_x^\alpha]W_2. \nonumber
\end{align}
Substituting these equalities in (\ref{m1.6}) and integrating with respect to the space variable, we obtain
\begin{align}\label{m1.7}
&\frac{1}{2} \frac{d}{dt} \int (\partial_x^\alpha W)^t A_0 \partial_x^\alpha  W dx +\int ( \partial_x^\alpha  W_2)^t B \partial_x^\alpha  W_2 dx= \int ( \partial_x^\alpha  W_2)^t [B, \partial_x^\alpha]W_2 dx \nonumber \\
& \qquad +\int (\partial_x^\alpha W)^t \left(A_0[\partial_x^\alpha, A_0^{-1}]G +\sum^n_{j=1} A_0 [\widetilde{A}_j, \partial_x^\alpha]  \partial_{x_j} W+ A_0[\partial_x^\alpha, A_0^{-1}]H\right) dx \nonumber \\
& \qquad + \frac{1}{2}\int (\partial_x^\alpha W)^t \left(  \partial_t A_0 + \sum^n_{j=1}  \partial_{x_j} A_j \right)\partial_x^\alpha W dx+ \int (\partial_x^\alpha W)^t \partial_x^\alpha H dx.
\end{align}

Let us analyze these integrals separately. Some of them can estimate by classical arguments, following the approach of \cite{HaNa,Yo}. \\
Above all $A_0$ and $B$ are positive definite, so there exist two constants such that
\begin{align}\label{m1.8}
\begin{array}{c}
	(\partial_x^\alpha W)^t A_0 \partial_x^\alpha W \ge c |\partial_x^\alpha W|^2, \qquad \qquad ( \partial_x^\alpha  W_2)^t B \partial_x^\alpha  W_2 \ge c |\partial_x^\alpha  W_2|^2.
\end{array} 
\end{align}
These inequalities allow us to estimate the left-hand side of (\ref{m1.7}).\\
Now, we can estimate the time integral of the first term on the right-hand side of (\ref{m1.7}) using Lemma (\ref{cal}). Indeed, thanks to the condition (ii) and the regularity of functions, we have
\begin{align}
\int_0^t	\int \left|( \partial_x^\alpha  W_2)^t [B, \partial_x^\alpha]W_2\right| dx d \tau
	\le c \int_0^t  \left\|B'\right\|_{H^{s-1}} \left(\left\|\partial_x^\alpha  W_2\right\|_{L^2}^2 + \left\|W_2\right\|_{H^{|\alpha|-1}}^2\right) d \tau.  \label{m1.9c}
\end{align}
Then, we consider the second and the third integral on the right-hand side. We know that $\partial_x^\alpha W \in H^{s-| \alpha|}$ and $A_0 \in H^s$, so $s_3:=\min \left\{s, s-|\alpha|, 2s- |\alpha|-s_0\right\}$ is a positive constant. Therefore, using the condition (i) of Lemma (\ref{cal}), we obtain that $A_0 \partial_x^\alpha W \in L^2$ and $\left\|A_0 \partial_x^\alpha W \right\|_{L^2} \le c \left\|\partial_x^\alpha W\right\|_{H^{s-|\alpha|}} \left\|A_0\right\|_{H^s}$. Then, using again condition (ii) of the same lemma and the regularity of functions, we deduce that $[\partial_x^\alpha, A_0]G \in L^2$ and $[\partial_x^\alpha, A_j] \partial_{x_j} W \in L^2$, for each $j=1,...,n$. As a consequence of these observations, we can calculate
\begin{align}
&\int_0^t \int \left|(\partial_x^\alpha W)^t  A_0[\partial_x^\alpha, A_0^{-1}]G \right|dx d\tau \nonumber \\
 &\qquad \le   c\int_0^t \left[\left\|\partial_x^\alpha W\right\|_{H^{s-|\alpha|}}^2 \left\|A_0\right\|_{H^{s}}^2 + \left\|(A_0^{-1})'\right\|_{H^{s-1}}^2 \left\|W_2\right\|^2_{H^{|\alpha|-1}}\right]d\tau, \label{m1.9a}
\end{align}
and
\begin{align}
&	\int_0^t \int \left|(\partial_x^\alpha W)^t A_0  \sum^n_{j=1} [\widetilde{A}_j, \partial_x^\alpha]  \partial_{x_j} W\right| dx d\tau \nonumber \\
 &\qquad
	 \le  c \int_0^t  \left\|A_0 \right\|_{H^s} \sum^n_{j=1} \left\|\widetilde{A}'_j\right\|_{H^{s-1}} \left(\left\|\partial_x^\alpha W\right\|_{H^{s-|\alpha|}}^2+\left\|\nabla W\right\|_{H^{| \alpha|-1}}^2\right) d\tau. \label{m1.9b}
\end{align}
Moreover, in the same way, we get
\begin{align}
	\int \left|(\partial_x^\alpha W)^t  A_0[\partial_x^\alpha, A_0^{-1}]H \right|dx 
  \le   C \left\|\partial_x^\alpha W\right\|_{H^{s-| \alpha |}} \left\|A_0\right\|_{H^s} \left\|(A_0^{-1})'\right\|_{H^{s-1}} \left\|H\right\|_{H^{|\alpha|-1}}. \nonumber
\end{align}
Let us point out that classical arguments are not a sucessful strategy to estimate the r.h.s. of the previous inequality. Since the last term causes the failure of standard approaches, our aim is to show an effective technique to estimate it.\\
So, we focus our attention on this term and we get the following inequality:
\begin{align}
	\left\|H(U+\overline{U},\nabla \phi)\right\|_{H^{|\alpha|-1}} 
	\le\mu \left[\left\|\rho \right\|_{L^\infty}  \left\|\nabla \phi \right\|_{H^{|\alpha|-1}}+ \left\|\rho \right\|_{H^{|\alpha|-1}}  \left\|\nabla \phi \right\|_{L^\infty} + \overline{\rho} \left\|\nabla \phi \right\|_{H^{|\alpha|-1}}\right]. \nonumber 
\end{align}
Now, substituting this estimate in the previous one and integrating with respect to $t$, we get 
\begin{align}
\int_0^t 	\int & \left|(\partial_x^\alpha W)^t  A_0[\partial_x^\alpha, A_0^{-1}]H  \right|dx d \tau \nonumber\\
\le&   c A \bigg[\sup_{0 \le \tau \le t}\left\|\rho \right\|_{L^\infty}  \int_0^t  \left(\left\|\partial_x^\alpha W\right\|_{H^{s-| \alpha |}}^2+  \left\|\nabla \phi \right\|^2_{H^{|\alpha|-1}}\right)d \tau \nonumber \\
& +\sup_{0 \le \tau \le t}   \left\|\rho(\tau) \right\|_{H^{|\alpha|-1}} \int_0^t \left(\left\|\partial_x^\alpha W\right\|^2_{H^{s-| \alpha |}}+ \left\|\nabla \phi \right\|^2_{L^\infty}\right) d\tau  \nonumber \\
& +\overline{\rho}\int_0^t  \left(\left\|\partial_x^\alpha W\right\|^2_{H^{s-| \alpha |}} + \left\|\nabla \phi \right\|^2_{H^{|\alpha|-1}}  \right)\bigg] \nonumber \\
\le&   c A \bigg[N_s(t) \left(N_s^2(t)+  \left\|\phi_0 \right\|^2_{H^{|\alpha|}}+ N_\alpha^2(t)\right) + N_{\alpha-1}(N_s^2(t)+\left\|\phi_0\right\|^2_{H^{s}}+N_s^2(t)) \nonumber \\
&+\overline{\rho} \left(N_s^2(t)+ \left\| \phi_0 \right\|^2_{H^{|\alpha|}}+ N_\alpha^2(t)  \right)\bigg] \nonumber \\
\le&   c A \left(N_s^3(t) +  c(\left\|\phi_0 \right\|_{H^{s}}) N_s(t)+ c(\overline{\rho})N_{s}^2(t)+ c(\left\| \phi_0 \right\|_{H^{|\alpha|}},\overline{\rho})\right), \label{h_a}
\end{align}
where $A:= \sup\limits_{0 \le \tau \le t} \left\|A_0\right\|_{H^s} \left\|(A_0^{-1})'\right\|_{H^{s-1}}$.\\
Next, the last integral of inequality (\ref{m1.7}) can be studied in the following way: 
\begin{align}
	 \int (\partial_x^\alpha W)^t \partial_x^\alpha H dx \le &\mu \int | ( \partial_x^\alpha  W_2)^t \partial_x^\alpha(\rho \nabla \phi)| dx+ \mu \overline{\rho} \int | ( \partial_x^\alpha  W_2)^t \partial_x^\alpha(\nabla \phi)| dx \nonumber \\
	 \le& \mu \| \partial_x^\alpha  W_2 \|_{L^2} \left(\| \rho \|_{L^\infty}\| \nabla \phi \|_{H^\alpha} + \| \rho\|_{H^\alpha} \|\nabla \phi \|_{L^\infty}\right)\nonumber\\
	 &+ \mu \overline{\rho} \| \partial_x^\alpha  W_2 \|_{L^2} \|\nabla \phi\|_{H^\alpha}, \nonumber	  
\end{align}
which, integrating with respect to the time variable, yields
\begin{align}
 \int_0^t	\int (\partial_x^\alpha W)^t \partial_x^\alpha H dx d\tau &  \le \mu \sup_{0 \le \tau \le t}\| \rho \|_{L^\infty}  \int_0^t \left( \| \partial_x^\alpha  W_2 \|^2_{L^2} + \| \nabla \phi \|^2_{H^\alpha}\right) d \tau  \nonumber \\
 & \quad  +  \mu \sup_{0 \le \tau \le t} \| \rho\|_{H^\alpha} \int_0^t \left(  \| \partial_x^\alpha  W_2 \|^2_{L^2}+ \|\nabla \phi \|^2_{L^\infty}\right) d \tau\nonumber\\
 & \quad+  \mu \overline{\rho}  \int_0^t  \left(\| \partial_x^\alpha  W_2 \|^2_{L^2}
 + \|\nabla \phi\|^2_{H^\alpha}\right) d \tau \nonumber \\
  &  \le c N_s^3(t)+c ( \| \phi_0 \|_{H^{s}},\overline{\rho}) N_s(t) + c (\overline{\rho})  N_\alpha^2(t)+c(\left\|\phi_0\right\|_{H^\alpha},\overline{\rho}).  \label{m1.9d} 
\end{align}

\begin{rmk}
	Let us point out that, in order to estimate the second integral of (\ref{m1.9d}), it is not useful to consider $\sup\limits_{0\le \tau \le t} \left\|\nabla \phi\right\|_{L^\infty}$. Indeed, it is impossible to estimate this term by the functional $N_s$, since, as deduced by (\ref{phix}), the order of the functional should be increased up to $s+1$. While, as noticed in Remark (\ref{stima_phix}), we can control the time integral of $\, \,\left\|\nabla \phi\right\|_{L^\infty} \,$ by $\,\,N_s$, without increasing the higher order derivative. 
\end{rmk}

Now, we examine the remaining term of (\ref{m1.7}). Using (\ref{A0W}) and the definition of $\widetilde{A}_j$, we can write
\begin{align}
	\partial_t A_0+ \sum^n_{j=1}  \partial_{x_j} A_j 
	&= -A'_0 \left(\sum^n_{j=1} \widetilde{A}_j \partial_{x_j} W \right)+A'_0 \left(A_0^{-1}G\right)+A'_0 \left(A_0^{-1}H\right)+\sum^n_{j=1} A'_j \partial_{x_j} W \nonumber \\
	&= \sum^n_{j=1} A_0\left(\widetilde{A}'_j \partial_{x_j} W\right) +A'_0 \left(A_0^{-1}G\right)+A'_0 \left(A_0^{-1}H\right). \nonumber
\end{align}

From this equality, recalling that $G=(0,-BW_2)^t$, we deduce
\begin{align}
\int 	\left|(\partial_x^\alpha W)^t \left(\partial_t A_0+ \sum^n_{j=1}  \partial_{x_j} A_j\right)\partial_x^\alpha W \right| d x 
	\le&  c \overline{A} \left(\sum^n_{j=1} \left\|\partial_{x_j} W \right\|_{L^2} + \left\|W_2\right\|_{L^2}  \right) \left\|\partial_x^\alpha W\right\|_{L^2}^2 \nonumber \\
	& +c \overline{A} \left({\left\|(\rho + \overline{\rho})\nabla \phi\right\|_{L^2}} \right) \left\|\partial_x^\alpha W\right\|_{L^2}^2 \nonumber \\
\le&  c \overline{A}\left( \right.\left\|W_2\right\|_{L^2}\left\|\partial_x^\alpha W\right\|_{L^2}^2+ \left\|\nabla W\right\|_{L^2}\left\|\partial_x^\alpha W\right\|_{L^2}^2\nonumber \\
	&+ {\left\|(\rho + \overline{\rho})\nabla \phi\right\|_{L^2} \left\|\partial_x^\alpha W\right\|_{L^2}^2} \left.\right),\qquad  \label{in.1}
\end{align}
where $\overline{A}:= \sup\limits_{0 \le \tau \le t} \left\{ \left\|A_0\right\|_{H^{s}} \sum\limits_{j=1}^n \left\|\widetilde{A}'_j\right\|_{H^{s-1}}+ \left\|A'_0\right\|_{H^{s-1}}\left\|A_0^{-1}\right\|_{H^{s}} (1+ \left\|B\right\|_{H^s})\right\}$. \\
Let us analyze these terms separately. First of all, we have
\begin{align}
	\int_0^t \left\|W_2\right\|_{L^2}\left\|\partial_x^\alpha W\right\|_{L^2}^2 d\tau &\le \sup_{0 \le \tau \le t} \left\|\partial_x^\alpha W(\tau)\right\|_{L^2}  \int_0^t \left(\left\|\partial_x^\alpha W\right\|_{L^2}^2 + \left\|W_2\right\|_{L^2}^2	\right) d \tau \nonumber \\
	&\le N_\alpha (t) (N_0^2(t)+N_{\alpha-1}^2(t)), \nonumber
\end{align}
and
\begin{align}
	\int_0^t \left\|\nabla W\right\|_{L^2}\left\|\partial_x^\alpha W\right\|_{L^2}^2 d\tau
	\le N_\alpha (t) (N_1^2(t)+N_{\alpha-1}^2(t)). \nonumber
\end{align}
Now, we are interested in studying the last term of the inequality (\ref{in.1}). Let us observe that
\begin{align}
	\int_0^t \left\|(\rho+\overline{\rho})\nabla \phi\right\|_{L^2} \left\|\partial_x^\alpha W \right\|_{L^2} d \tau \le & \sup_{0 \le \tau \le t}\left\| \rho\right\|_{L^2} \int_0^t  \left(\left\|\nabla \phi \right\|^2_{L^\infty}+  \left\|\partial_x^\alpha W \right\|_{L^2}^2\right)d \tau  \nonumber \\
	&+ \sup_{0 \le \tau \le t} \left\|{\rho} (\tau) \right\|_{L^\infty} \int_0^t \left(\left\|\nabla \phi\right\|_{L^2}^2+  \left\|\partial_x^\alpha W \right\|_{L^2}^2 \right) d \tau\nonumber\\
	&+ \overline{\rho} \int_0^t \left(\left\|\nabla \phi\right\|_{L^2}^2 +  \left\|\partial_x^\alpha W \right\|_{L^2}^2 \right) d \tau \nonumber \\	
\le & c (\left\|\phi_0\right\|_{H^{s}})N_s(t)+ c(\overline{\rho}) N_\alpha^2(t) + N_s^3(t)+ c( \left\|\phi_0\right\|_{L^2}, \overline{\rho}). \label{m1.9e}
\end{align}
Finally, integrating the equation (\ref{m1.7}) with respect to the time variable and substituting in it inequalities (\ref{m1.8}), (\ref{m1.9c}), (\ref{m1.9a}), (\ref{m1.9b}), (\ref{h_a}), (\ref{m1.9d}) and (\ref{m1.9e}), we deduce
\begin{align}
\| \partial_x^\alpha W(t)\|_{L^2}^2+\int_0^t\|\partial_x^\alpha W_2(\tau)\|_{L^2}^2 d\tau 
\le& c(\left\|W_0\right\|_{H^s}, \left\|\phi_0\right\|_{H^s},\overline{\rho})+ c(\left\|\phi_0\right\|_{H^{s}},\overline{\rho}) M_s(t) N_s(t)\nonumber\\
& + c(\overline{\rho}) M_s(t) N_s^2(t) + c M_s(t) N_s^3(t), \nonumber   
\end{align}
where
\begin{align}
	M_s(t):=\sup\limits_{0 \le \tau \le t }  &\Bigg[\left\|A_0\right\|_{H^s}^2+\left\|(A_0^{-1})'\right\|_{H^{s}}^2+ \left\|A_0\right\|_{H^{s}} \sum\limits_{j=1}^n \left\|\widetilde{A}'_j\right\|_{H^{s-1}}+ \left\|A_0\right\|_{H^{s}} \left\|(A_0^{-1})'\right\|_{H^{s-1}}\nonumber \\
	&+ \left\|B'\right\|_{H^{s-1}}+ \left\|A'_0\right\|_{H^{s-1}} \left\|(A_0^{-1})\right\|_{H^{s}} (1+ \left\|B\right\|_{H^s})\nonumber \\
	&+(1+ \left\|W\right\|_{H^{s-1}}+ \left\|W\right\|_{H^{s-1}}^{s-1})^2 \left\|B\right\|^2_{H^{s-1}}+1+ \left\|W\right\|_{H^{s-1}}+ \left\|W\right\|_{H^{s-1}}^{s-1}\nonumber \\
	&+ \Bigg(\sum\limits_{j=1}^n \left\|\widetilde{A}_j(0)-\widetilde{A}_j(W)\right\|_{H^{s-1}}\Bigg)^2\Bigg].\label{costante_Ms}
\end{align}
Therefore, summing up for $1 \le |\alpha| \le s$, we deduce the following $s$-order estimate of function $W$
\begin{align}
\| W(t)\|_{H^s}^2+\int_0^t\| W_2(\tau)\|_{H^s}^2 d\tau\leq &   C(\left\|W_0\right\|_{H^s},\left\|\phi_0\right\|_{H^s},\overline{\rho})+ C( \left\|\phi_0\right\|_{H^{s}}, \overline{\rho}) M_s(t) N_s(t)\nonumber \\
&+ C(\overline{\rho}) M_s(t) N_s^2(t)+ C M_s(t) N_s^3(t).\nonumber 
\end{align}


\subsection{Proof of the Global Existence Theorem}
Now, we are finally able to prove Theorem (\ref{global}), showing the existence of a global smooth solution for system (\ref{variabili_entro}).  

\begin{proof}
Let us recall the definition of the functionals
\begin{align}
	N_l^2(t)&:= \sup_{0 \le \tau \le t} \left\|W(\tau)\right\|_{H^l}^2 + \int_0^t \left\|W_2(\tau)\right\|_{H^l}^2 d\tau + \int_0^t \left\|\nabla W(\tau)\right\|^2_{H^{l-1}}d \tau,\; \textrm{for } l=1,...,s, \nonumber\\
	N_0^2(t)&:= \sup_{0 \le \tau \le t} \left\|W(\tau)\right\|_{L^2}^2 + \int_0^t \left\|W_2(\tau)\right\|_{L^2}^2 d\tau,  \nonumber
\end{align}
and the energy estimates, obtained in previous sections, 
\begin{align}
\|W(t)\|^2_{L^2}+\int_0^t\|W_2(\tau)\|^2_{L^2}d\tau \leq &N_0^2(0)+ C(\left\| \phi_0\right\|_{L^2}) N_1(t) + C(\overline{\rho}) N_1^2(t)\nonumber\\
 &+C N_1^3(t)+ C(\left\|\phi_0 \right\|_{L^2},\overline{\rho})\label{1.N1},
\end{align}
and, for each $s\ge1$,
\begin{align}
\| W(t)\|_{H^s}^2+\int_0^t\| W_2(\tau)\|_{H^s}^2 d \tau \leq&   C(\left\|W_0\right\|_{H^s},\left\|\phi_0\right\|_{H^s},\overline{\rho})+ C(\left\|\phi_0\right\|_{H^{s}}, \overline{\rho}) M_s(t) N_s(t)\nonumber \\
&+ C(\overline{\rho}) M_s(t) N_s^2(t) 
+ C M_s(t) N_s^3(t).\label{1.N2}
\end{align}
Therefore, to obtain an estimate of the functional $N^2_s(t)$, we have to study also the term 
$$
\int_0^t\|\nabla W(\tau)\|^2_{H^{l-1}}d \tau,\qquad  \quad \textrm{for } l=1,...,s.
$$
To this end, we rewrite the first equation of system (\ref{variabili_entro}), in the following way
\begin{align}
	 \partial_t W + \sum_{j=1}^n \widetilde{A}_j(0) \partial_{x_j} W= A_0^{-1}(W) G(W) +A_0^{-1}(W) H(W, \nabla \phi)+{L}(W, \partial_x W),\nonumber
\end{align}	 
where $L:=\sum\limits^n_{j=1}\left( \widetilde{A}_j(0)-\widetilde{A}_j(W) \right)\partial_{x_j} W$. Applying the Fourier transform with respect to $x$, we obtain 
\begin{align}
	 \partial_t \widehat{W}+i \sum^n_{j=1} \xi_j \widetilde{A}_j(0) \widehat{W} =\widehat{A_0^{-1}G}+\widehat{A_0^{-1}H}+ \widehat{L}. \label{sist_fourier}
\end{align}
Let us recall that, in Section (\ref{SKcond}), we proved that the first equation of system (\ref{variabili_entro}) without the term $H(W, \nabla \phi)$ satisfies the condition (SK). As shown by Shizuta and Kawashima \cite{ShiKa}, this means that there exist a constant $c>0$ and a skew-symmetric real matrix $K=K(\xi) \in {C}^\infty (S^{n-1})$ satisfying $K(-\xi)=-K(\xi)$ and 
\begin{align}
	\frac{1}{2}\left[K(\xi)\widetilde{A}(\xi)+\left(K(\xi)\widetilde{A}(\xi)\right)^t\right] + |\xi| \textrm{diag}(0,I_{n}) \ge c |\xi| I_{n+1}, \label{m1.13}
\end{align}
for every $\xi \in S^{n-1}$, where $S^{n-1}$ is the unit sphere in $\R^n$ and 
$$	\widetilde{A}(\xi):= \sum^n_{j=1} \widetilde{A}_j(0) \xi_j, \qquad \qquad \xi \in \R^n \backslash \left\{0\right\}.$$ 
Now, if we multiply the system (\ref{sist_fourier}) by $-i\widehat{W}^t K$, then we have
\begin{eqnarray}
 -i\widehat{W}^t K \partial_t\widehat{W}+\widehat{W}^t K \sum^n_{j=1} \xi_j \widetilde{A}_j(0) \widehat{W} = -i\widehat{W}^t K (\widehat{A_0^{-1}G}+\widehat{A_0^{-1}H}+\widehat{L}).	\nonumber
\end{eqnarray}
Substituting inequality (\ref{m1.13}) 
and 
\begin{eqnarray}
	2 \textrm{Im} \widehat{W}^t K(\widehat{A_0^{-1}G}+\widehat{A_0^{-1}H}+\widehat{L}) \le c|\xi| | \widehat{W}|^2+C|	\xi|^{-1}(|\widehat{A_0^{-1}G}|^2+|\widehat{A_0^{-1}H}|^2+|\widehat{L}|^2) \nonumber
\end{eqnarray}
in the previous system, we obtain 
\begin{eqnarray}
	-i 	\partial_t \left(\widehat{W}^t K\widehat{W}\right)  +c|\xi||\widehat{W}|^2 \le 2 |\xi| |\widehat{W_2}|^2 + C |	\xi|^{-1}(|\widehat{A_0^{-1}G}|^2+|A_0^{-1}\widehat{H}|^2+|\widehat{L}|^2). \nonumber	
\end{eqnarray}
Let us multiply this last inequality by $|\xi|^{2k-1}$, with $k \ge 1$, and integrate  
over $\R^n \times [0,t]$, so we obtain  
\begin{align}
c \int_0^t \int | \xi|^{2k}|\widehat{W}|^2 d \xi d \tau  \le &2 \int_0^t \int |\xi|^{2k} |\widehat{W_2}|^2 d\xi d\tau \nonumber\\
&+C \int | \xi|^{2k-1}|\widehat{W}(\xi, t)|^2 d \xi +C \int | \xi|^{2k-1}|\widehat{W}_0|^2 d \xi  \nonumber\\
&  + C \int_0^t \int |	\xi|^{2k-2}( |\widehat{A_0^{-1}G}|^2+|\widehat{A_0^{-1}H}|^2+|\widehat{L}|^2)d\xi d\tau. \nonumber
\end{align}
Then, since $2 | \xi| \le 1+ | \xi |^2$ and $k \ge 1$, we deduce that
\begin{align}
 \int_0^t \sum_{|\alpha|=k} \left\|\partial_x^\alpha W(\tau)\right\|_{L^2}^2 d\tau  
\le& c \left[ \int_0^t \left(\sum_{|	\alpha|=k} \left\|\partial_x^\alpha  W_2(\tau) \right\|_{L^2}^2 \right)d\tau + \left\|{W}(t)\right\|_{H^k}^2 + \left\|{W}_0\right\|_{H^k}^2 \right. \nonumber\\
 &+ \int_0^t \sum_{|\alpha|=k-1}\left(\left\|\partial_x^\alpha  (A_0^{-1}G)\right\|_{L^2}^2+\left\|\partial_x^\alpha  (A_0^{-1}H)\right\|_{L^2}^2 d\tau\right),\nonumber\\
& +\left.\int_0^t \sum_{|\alpha|=k-1}\left\|\partial_x^\alpha L\right\|_{L^2}^2d \tau\right],\nonumber
\end{align}
which, summing over all $\alpha$ such that $|\alpha| \in [1,s]$, yields 
\begin{align}
 \int_0^t \left\|\nabla W(\tau) \right\|_{H^{s-1}}^2 d \tau  
\le& c\left[ \int_0^t \left\|W_2(\tau) \right\|_{H^s}^2 d \tau+ \left\|{W}(t)\right\|_{H^s}^2 + \left\|{W}_0\right\|_{H^s}^2 \right.  \nonumber\\
 &+ \left. \int_0^t \left( \left\|A_0^{-1}G\right\|_{H^{s-1}}^2+\left\|A_0^{-1}H \right\|_{H^{s-1}}^2+\left\|L \right\|_{H^{s-1}}^2 \right) d\tau\right]. \nonumber
\end{align}
Now, let us recall that $L=\sum\limits^n_{j=1}\left( \widetilde{A}_j(0)-\widetilde{A}_j(W) \right)\partial_{x_j} W$, so using condition (i) of Lemma (\ref{cal}), we get
\begin{align}
	\left\|L\right\|_{H^{s-1}} \le& c \sum^n_{j=1}\left\|\widetilde{A}_j(0)- \widetilde{A}_j(W)\right\|_{H^{s-1}} \left\|\partial_{x_j} W \right\|_{H^{s-1}}, \nonumber 
\end{align}	
then
\begin{eqnarray}
\int_0^t	\left\|L\right\|^2_{H^{s-1}} d \tau &\le&	 c M_s (t) N_{s}^2(t), \nonumber 
\end{eqnarray}	
where $M_s(t)$ is defined by (\ref{costante_Ms}).\\ 
Using again Lemma (\ref{cal}), we deduce that 
\begin{align}
\left\| A_0^{-1}G\right\|_{H^{s-1}} &\le \left\| A_0^{-1}(0)G\right\|_{H^{s-1}} + \left\|\left[ A_0^{-1}(W)- A_0^{-1}(0)\right]G\right\|_{H^{s-1}} \nonumber \\
	&\le c \left(1 + \left\| A_0^{-1}(W)- A_0^{-1}(0)\right\|_{H^{s-1}}\right)\left\|G\right\|_{H^{s-1}} \nonumber \\
		&\le c (1+\left\|W\right\|_{H^{s-1}}+\left\|W\right\|^{s-1}_{H^{s-1}})\left\|B\right\|_{H^{s-1}} \left\|W_2\right\|_{H^{s-1}}, \nonumber
\end{align}
which yields
\begin{align}
\int_0^t \left\| A_0^{-1}G\right\|^2_{H^{s-1}} d \tau \le& c M_s(t) N_{s-1}^2(t). \nonumber
\end{align}
Proceeding in the same way, we get
\begin{eqnarray}
\left\| A_0^{-1}H\right\|_{H^{s-1}} \le c (1+\left\|W\right\|_{H^{s-1}}+\left\|W\right\|^{s-1}_{H^{s-1}})\left\|H\right\|_{H^{s-1}}, \nonumber
\end{eqnarray}
so, we have 
\begin{align}
	\int_0^t \left\| A_0^{-1}H\right\|^2_{H^{s-1}} d \tau 
	 \le&  M_s(t) \int_0^t \left( \left\|\rho \right\|_{L^\infty} \left\|\nabla \phi\right\|_{H^{s-1}}+ \left\|\rho\right\|_{H^{s-1}} \left\|\nabla \phi\right\|_{L^\infty}+ \overline{\rho} \left\|\nabla \phi\right\|_{H^{s-1}}\right)^2 d\tau \nonumber \\
	 \le&  M_s(t) \bigg[ \sup_{0\le \tau \le t} \left\|\rho \right\|^2_{L^\infty}  \int_0^t \left\|\nabla \phi\right\|^2_{H^{s-1}}d\tau\nonumber \\
	 &+  \sup_{0\le \tau \le t}   \left\|\rho\right\|^2_{H^{s-1}} \int_0^t \left\|\nabla \phi\right\|^2_{L^\infty}d\tau + \overline{\rho}^2 \int_0^t \left\|\nabla \phi\right\|^2_{H^{s-1}}d\tau\nonumber \\
	 & + 2 \sup_{0\le \tau \le t} \left\|\rho \right\|_{L^\infty}  \sup_{0\le \tau \le t}\left\|\rho\right\|_{H^{s-1}} \int_0^t \left(\left\|\nabla \phi\right\|^2_{H^{s-1}} + \left\|\nabla \phi\right\|^2_{L^\infty}\right)d\tau\nonumber \\ 
& + 2 \overline{\rho} \sup_{0\le \tau \le t} \left\|\rho \right\|_{L^\infty} \int_0^t \left\|\nabla \phi\right\|^2_{H^{s-1}} d\tau \nonumber \\
&+ 2  \overline{\rho} \sup_{0\le \tau \le t}  \left\|\rho\right\|_{H^{s-1}} \int_0^t\left( \left\|\nabla \phi\right\|^2_{L^\infty}+ \left\|\nabla \phi\right\|^2_{H^{s-1}}\right) d\tau \bigg]	\nonumber \\
   \le&  M_s(t) \bigg[N_s^2(t) \left(  \left\|\phi_0\right\|^2_{H^{s}}+N_{s-1}^2(t)\right)+ N_{s}^2(t) \left(  \left\|\phi_0\right\|^2_{H^{s}}+N_{s}^2(t)\right) \nonumber \\
&+ \overline{\rho}^2 \left(  \left\|\phi_0\right\|^2_{H^{s}}+N_{s}^2(t)\right)+ 2 N_s^2(t)\left(  \left\|\phi_0\right\|^2_{H^{s}}+N_{s}^2(t)\right)\nonumber \\
&+  2 N_s^2(t) \left(  \left\|\phi_0\right\|^2_{H^{s}}+N_{s}^2(t)\right) + 2 \overline{\rho} N_s(t) \left(  \left\|\phi_0\right\|^2_{H^{s}}+N_{s}^2(t)\right)\nonumber \\ 
&+ 2  \overline{\rho} N_{s}(t) \left(  \left\|\phi_0\right\|^2_{H^{s}}+N_{s}^2(t)\right) +2  \overline{\rho} N_{s-1}(t)  \left(  \left\|\phi_0\right\|^2_{H^{s}}+N_{s}^2(t)\right)\bigg] \nonumber\\
   \le& c(\left\|\phi_0\right\|_{H^{s}},\overline{\rho}) M_s(t) N_s^2(t) + M_s(t) N_{s}^4(t) + c(  \left\|\phi_0\right\|_{H^{s}}, \overline{\rho}, M_s(t) )\nonumber \\ 
 & + c( \left\|\phi_0\right\|_{H^{s}},\overline{\rho})  M_s(t) N_s(t) + c(\overline{\rho})  M_s(t) N_{s}^3(t). \nonumber 
\end{align}
Consequently, as long as $M_s(t) \le C$, we obtain 
\begin{align}
  \int_0^t \left\|\nabla W \right\|_{H^{s-1}}^2 d \tau  
\le & c N_s^2(0)+    c(\left\|\phi_0\right\|_{H^{s}},\overline{\rho})  M_s(t) N_s(t) + c(\left\|\phi_0\right\|_{H^{s}},\overline{\rho}) M_s(t) N_s^2(t)\nonumber \\
&+   c(\overline{\rho}) M_s(t)  N_{s}^3(t)  +c M_s(t) N_{s}^4(t) + c(\left\|\phi_0\right\|_{H^{s}},  \overline{\rho}, M_s(t) ). \nonumber 
\end{align}
Combining the previous inequality with (\ref{1.N1}), (\ref{1.N2}), we get the estimate 
\begin{align}
	N_s^2(t) \le& \,C N_s^2(0) + C(\left\|\phi_0\right\|_{H^{s}},\overline{\rho}, M_s(t))+C (\left\|\phi_0\right\|_{H^{s}},\overline{\rho},M_s(t)) N_s(t)\nonumber \\ 
	&+ C(\overline{\rho},M_s(t)) N_s^2(t)+ C(\overline{\rho},M_s(t))N_s^3(t) + C( M_s(t) ) N_s^4(t). \nonumber  
\end{align}
In conclusion, choosing small initial data and small constant state, from the previous inequality we deduce the theorem, by classical arguments.\hfill
\end{proof}


\section{Asymptotic Behavior}
In this section we study the time decay properties of the global smooth solution to system (\ref{generale_perturbato1}), proceeding along the lines of \cite{BiHaNa}. Thanks to the decomposition of the Green function of the linearized problem, we aim to obtain the $H^s$ and $L^\infty$ decay estimates of the solution for the considered model.\\
To this end we rewrite system (\ref{generale_perturbato1}) in the Conservative-Dissipative form as
\begin{align}
\left\{
\begin{array}{l}
	\partial_t (U+ \overline{U})+ \sum\limits_{j=1}^n \partial_{x_j}f_j(U+\overline{U})=g(U)+ h(U+\overline{U}, \nabla \phi), \\ \\
	\partial_t \phi =D \Delta \phi +a \rho -b{\phi},
\end{array}
\right.\label{sistAs}
\end{align}
where
\begin{align*}
	U=\left( \begin{array}{c}
	\rho\\\frac{v}{\sqrt{P'(\overline{\rho})}}
\end{array}
\right),
\quad
	\overline{U}=\left( \begin{array}{c}
	\overline{\rho}\\0
\end{array}
\right),
\quad f_j(U+\overline{U})=\left(
\begin{array}{c}
\sqrt{P'(\overline{\rho})}	v_j \\ 
\sqrt{P'(\overline{\rho})} \frac{v_j v}{\rho+\overline{\rho}} + \frac{P(\rho+\overline{\rho})}{\sqrt{P'(\overline{\rho})}}e_j
\end{array}
\right), 
\end{align*}
\begin{align*}
g(U)=\left(
\begin{array}{c}
	0 \\ - \alpha v
\end{array}
\right), \quad
h(U+\overline{U},\nabla \phi)=\left(
\begin{array}{c}
	0 \\ \mu\frac{\rho+\overline{\rho}}{\sqrt{P'(\overline{\rho})}}\nabla \phi
\end{array}
\right).\nonumber
\end{align*}
Defined $\overline{f}_j(U)=f_j(U+\overline{U})-f_j(\overline{U})$ and $\overline{\mu}=\frac{\mu}{\sqrt{P'(\overline{\rho})}}$, the system can be rewritten in the following way
\begin{align}
	\partial_t U+  \sum\limits_{j=1}^n \partial_{x_j}  \left(\overline{f}_j'(\overline{U})U\right)=g({U}) + \sum_{j=1}^n\partial_{x_j} \left(\overline{f}_j'(\overline{U})U- \overline{f}_j(U)\right)+h(U+\overline{U},\nabla \phi), \label{mod4}
\end{align}
and its solution is given by
\begin{align}
U(t)=& \Gamma^h(t) *U_0+  \sum\limits_{j=1}^n  \int_0^t \partial_{x_j} \Gamma^h (t-\tau)* \left[\overline{f}_j'(\overline{U})U(\tau)-\overline{f}_j(U(\tau))\right]d\tau\nonumber \\
&+ \int_0^t \Gamma^h (t-\tau)*h(U+\overline{U},\nabla \phi)d\tau,  \label{3}	
\end{align}
where $\Gamma^h$ denotes the Green function of the linearized system
$$
\partial_t U+\sum\limits_{j=1}^n \overline{f}_j'(\overline{U})\partial_{x_j} U=g(U).
$$

Let us briefly recall the results on the Green Kernel of multidimensional dissipative hyperbolic systems obtained by Bianchini et al. in \cite{BiHaNa}. In their work the authors analyzed the behavior of the function $\Gamma^h(x,t)$ for linearized problems. It has been decomposed into two main terms: the diffusive one consisting of heat kernel and a faster term consisting of the hyperbolic part. In general, the form of the Green function is not explicit, but it is possible to deal with its Fourier transform. The separation of the Green Kernel into various parts is done at the level of a solution operator $\Gamma^h(t)$ acting on $L^1(\R^{n}) \cap L^2(\R^n)$.

They deeply described the behavior of the diffusive part, which is decomposed in four blocks, decaying with different rates. They showed that solutions
have canonical projections on two different components: the conservative part and the dissipative part. The first one, which formally corresponds to the conservative part of equations, decays in time like the heat kernel, since it corresponds to the diffusive part of the Green function. On the
other side, the dissipative part is strongly influenced by the dissipation and decays at a rate $t^{-\frac{1}{2}}$ faster than the conservative one.

They considered the Cauchy problem for the linear system in the conservative-dissipative form
\begin{equation}
\partial_t w+\sum_{j=1}^n A_j \partial_{x_j} w= Bw,\nonumber
\end{equation}
and they showed that it is possible to decompose the solution as
$$
w(t)=\Gamma^h(t)\ast w_0=K(t) w_0+\mathcal{K}(t)  w_0,
$$
for any function $w_0\in   L^1(\R^n) \cap L^2 (\R^n)$, where $K(t)$ is the diffusive part and $\mathcal{K}(t)$ is the trasport dissipative one. \\
Moreover for any multi index $\beta$ and for every $p\in [2,+\infty]$ the following estimates hold:

\begin{align*}
\|D^{\beta}\mathcal{K}(t) w_0\|_{L^2}\,\,\leq& \,\,C e^{-ct}\|D^{\beta}w_0\|_{L^2},\\ \\
\|L_0 D^\beta K(t) w_0\|_{L^p}\,\,\leq& \,\, C(|\beta|)\min\{1, t^{-\frac{m}{2}\left(1-\frac{1}{p}\right)-\frac{|\beta|}{2}}\}\|L_0 w_0\|_{L^1}\\
&+C(|\beta|)\min\{1, t^{-\frac{m}{2}\left(1-\frac{1}{p}\right)-\frac{1}{2}-\frac{|\beta|}{2}}\}\|L_{-} w_0\|_{L^1},\\\\
\|L_{-} D^\beta K(t) w_0\|_{L^p}\,\,\leq&\,\, C(|\beta|)\min\{1, t^{-\frac{m}{2}\left(1-\frac{1}{p}\right)-\frac{1}{2}-\frac{|\beta|}{2}}\}\|L_0 w_0\|_{L^1}\\
&+C(|\beta|)\min\{1, t^{-\frac{m}{2}\left(1-\frac{1}{p}\right)-1-\frac{|\beta|}{2}}\}\|L_{-} w_0\|_{L^1},
\end{align*}
where $L_0=[I_1,0]$ and $L_{-}=[0,I_2]$ are the projectors on the null space and on the negative definite part of $B$.

\subsection{$H^s$ Estimates of the Solution}
This section is devoted to study the decay rates of solution to the system (\ref{sistAs}) in the $H^s$-norm.\\
We define
\begin{align}
	E_s := \max \left\{ \left\|U_0\right\|_{L^1}, \, \left\|U_0\right\|_{H^s} \right\}, \quad\qquad D_s := \max \left\{ \left\|\phi_0\right\|_{L^1}, \, \left\|\phi_0\right\|_{H^{s}} \right\}, \nonumber
\end{align}
and the general functional 
\begin{align}
	S_w^{\alpha} := \sup_{0 \le \tau \le t} \left\{ \max \left\{1, \tau^{\alpha} \right\} \left\|w(\tau)\right\|_{H^s} \right\}. \nonumber
\end{align}
Then, we shall prove the following theorem 
\begin{theorem}\label{Deca_Hs}
Let $(U,\phi)$ be a global solution to problem (\ref{sistAs}), with initial conditions
$$
U(x,0)=U_0(x),\qquad \phi(x,0)=\phi_0(x),
$$
with
$$
U_0\in H^{s+1}(\R^n)\cap L^1(\R^n), \qquad \quad\phi_0\in H^{s+1}(\R^n)\cap L^1(\R^n),\, \qquad  \mbox{ for } s > \left[\frac{n}{2}\right]+ 1.
$$
Then the following decay estimates hold:
\begin{align}
&	\|U(t)\|_{H^s}  \leq \min\{1, t^{-\frac{n}{4}}\}C(E_{s+1},D_{s+1}, \overline{\rho}), \nonumber \\
& 	\|\phi(t)\|_{H^{s+1}} \leq \min\{1, t^{-\frac{n}{4}}\}C(E_{s+1}+D_{s+1},\overline{\rho}). \nonumber 
\end{align}
\end{theorem}

\begin{proof}
First we consider the parabolic equation
$$
\partial_t \phi =D\Delta\phi +a u -{b\phi},
$$
and, using the Duhamel's formula, we can write the solution as
\begin{equation}
\phi(x,t)=(e^{-{bt}}\Gamma^p(t)\ast \phi_0)(x)+\int_0^t e^{-{b(t-\tau)}}\Gamma^p(t-\tau)\ast a\rho(\tau) d\tau,\nonumber 
\end{equation}
where
$$
\Gamma^p(x,t):=\frac{e^{-\frac{|x|^2}{4 Dt}}}{(4 \pi Dt)^{n/2}}.
$$
Let us start with the $H^{s+1}$ estimate:
\begin{align*}
\|\phi(t)\|_{H^{s+1}}
\leq & ce^{-bt}\|\phi_0\|_{H^{s+1}}+c\int_0^t e^{- b(t-\tau)}\|a\rho(\tau)\|_{L^2}d\tau\nonumber \\
&+c\int_0^t e^{- b(t-\tau)}(t-\tau)^{-\frac{1}{2}}\|a\rho(\tau) \|_{H^s}d\tau\\
\leq& c e^{-b t}\|\phi_0\|_{H^{s+1}}+c S_{U}^{\frac{n}{4}}(t)\int_0^t e^{- b(t-\tau)}(t-\tau)^{-\frac{1}{2}}\min\{1,\tau^{-\frac{n}{4}}\}d\tau  \nonumber \\ 
&+c S_{U}^{\frac{n}{4}}(t)\int_0^t e^{- b(t-\tau)}\min\{1,\tau^{-\frac{1}{4}}\}d\tau .
\end{align*}

So we obtain the following $H^{s+1}$ estimate for $\phi$
\begin{equation*}
\|\phi(t)\|_{H^{s+1}}\leq c(e^{- bt}\|\phi_0\|_{H^{s+1}}+\min\{1,|t-1|^{-\frac{n}{4}}\}S_{U}^{\frac{n}{4}}(t)+\min\{1,t^{-\frac{n}{4}}\}S_{U}^{\frac{n}{4}}(t)),
\end{equation*}
which yields
\begin{equation}\label{Mphi}
S_{\phi_x}^{\frac{n}{4}}(t) \le  C(e^{- bt} \max\{1, t^{\frac{n}{4}}\} \|\phi_0\|_{H^{s+1}}+S_{U}^{\frac{n}{4}}(t)).
\end{equation}
Let us notice that from the previous inequality the decay rate of the function $\phi$ in $H^{s+1}$ is the same rate of the function $U$ in $H^s$.\\
Proceeding in a similar way, it is possible to get the following estimate for the function $\phi$ in the space $L^1$:

\begin{equation}\label{phiL1}
\|\phi(t)\|_{L^1}\leq e^{-bt}\|\phi_0\|_{L^1}+ c \sup_{\tau \in (0,t)}\|\rho(\tau)\|_{L^1},
\end{equation}
where, thanks to the mass conservation, $\sup\limits_{\tau \in (0,t)}\|\rho(\tau)\|_{L^1}=\|\rho_0\|_{L^1}$.

Now we focus on the estimate of function $U$. Let us observe that $\overline{f}_j(U)-\overline{f}_j'(\overline{U})U= U^2r_j(U)$ (where the product should be intended as the tensor product), therefore, using (\ref{3}) and the definition of $E_s$, we obtain
\begin{align}
	\left\|U(t)\right\|_{H^s} \le & c \min\{1, t^{-\frac{n}{4}}\} \left\|U_0\right\|_{L^1} +c e^{-ct} \left\|U_0\right\|_{H^s}\nonumber \\
	&+c \int_0^t \min\{1, (t-\tau)^{-\frac{n}{4}-\frac{1}{2}}\}\sum\limits_{j=1}^{n} \left\|U^2(\tau)r_j(U(\tau))\right\|_{L^1}d\tau  \nonumber\\%
	&+ c\int_0^t  e^{-c(t-\tau)} \sum\limits_{j=1}^n\left\| \partial_{x_j} (U^2(\tau)r_j(U)(\tau)) \right\|_{H^s}d\tau \nonumber \\
	&+ \int_0^t \|\Gamma^h (t-\tau)*h(U+\overline{U},\nabla \phi)(\tau)\|_{H^s} d\tau. \label{4}
\end{align}
At this stage we want to estimate the right hand side of this inequality.

Let us start studying the first integral in (\ref{4}), as follows
\begin{align}
	\int_0^t &\min\{1, (t-\tau)^{-\frac{n}{4}-\frac{1}{2}}\}\sum\limits_{j=1}^n \left\|U^2(\tau)r_j(U(\tau))\right\|_{L^1} d\tau \nonumber \\
	&\le	\int_0^t \min\{1, (t-\tau)^{-\frac{n}{4}-\frac{1}{2}}\}\|U(\tau)\|_{_{L^2}}^2\sum\limits_{j=1}^n\|r_j(U(\tau))\|_{_{L^\infty(|U|\leq \delta_0)}}d\tau \nonumber\\
	&\le	c (S_U^{\frac{n}{4}}(t))^2\int_0^t \min\{1, (t-\tau)^{-\frac{n}{4}-\frac{1}{2}}\} \min\{1,\tau^{-\frac{n}{2}}\} d\tau. \nonumber
\end{align}
Then from Lemma 5.2 of \cite{BiHaNa}, we deduce 
\begin{align}
	c\int_0^t &\min\{1, (t-\tau)^{-\frac{n}{4}-\frac{1}{2}}\} \left\|U^2(\tau)r_j(U)\right\|_{L^1} d\tau \nonumber \\
	&	\le	c \int_0^t \min\{1, (t-\tau)^{-\frac{n}{4}-\frac{1}{2}}\} \min\{1,\tau^{-\frac{n}{2}}\} (S_U^{\frac{n}{4}}(t))^2  \nonumber \\
	&\le c\min\{1,t^{-\nu} \}(S_U^{\frac{n}{4}}(t))^2, \label{5}
\end{align}
where $\nu=\min \left\{\frac{n}{4}+\frac{1}{2},\frac{n}{2},\frac{3}{4}n-\frac{1}{2}\right\}$.\\
In order to estimate the next term in (\ref{4}), we use Lemma 5.3 of \cite{BiHaNa} which yields
\begin{align}
\sum\limits_{j=1}^n	\left\|\partial_{x_j} (U^2r_j(U))\right\|_{H^s} &\le \sup_{j=1,...,n}c(\delta_0,\|u\|_{H^s},\|r_j\|_{C^{s+|\beta|}(|u|\leq\delta_0)})\left\|U\right\|_{L^\infty} \sum_{j=1}^n \left\|\partial_{x_j} U\right\|_{H^s} \nonumber\\
	&\le   c \left\|U\right\|_{H^s} \left\| U\right\|_{H^{s+1}}\label{eq4}.
\end{align}
Then we have
\begin{align}
\int_0^t e^{-c(t-\tau)} \sum\limits_{j=1}^n \left\|\partial_{x_j} (U^2r_j(U))(\tau)\right\|_{H^s} 
 \le &  c\int_0^t e^{-c(t-\tau)} \left\|U(\tau)\right\|_{H^s} \left\| U(\tau) \right\|_{H^{s+1}} d\tau \nonumber \\
 \le &  cS_U^{\frac{n}{4}} (t)E_{s+1}\int_0^t e^{-c(t-\tau)} \min\{1,\tau^{-\frac{n}{4}}\}  d\tau \nonumber \\
 \le &  c  \min\{1,t^{-\frac{n}{4}}\} S_U^{\frac{n}{4}} (t)E_{s+1} . \label{6}
\end{align}
In the last inequalities, we have used Lemma 5.2 of \cite{BiHaNa} and the estimate of Theorem (\ref{global}) to controll the norm of the function $U$ in $H^s$.

Finally we estimate the last integral of (\ref{4}) in the following way
\begin{align*}
	\int_0^t \|\Gamma^h (t-\tau)*h(U+\overline{U},\nabla \phi)(\tau)\|_{H^s}d\tau\leq &\int_0^t \|\mathcal{K}(t-\tau) h(U+\overline{U},\nabla \phi)(\tau)\|_{H^s}d\tau \\
	&+\int_0^t\|K(t-\tau) h(U+\overline{U},\nabla \phi)(\tau)\|_{H^s}d\tau.
\end{align*}
For the first term, we have:
\begin{align}
	\int_0^t \|\mathcal{K}(t-\tau) h(U+\overline{U},\nabla \phi)(\tau)\|_{H^s}d\tau
	\leq&\int_0^t ce^{-c(t-\tau)}\|\nabla \phi(\tau)\|_{H^s}(\overline{\rho}+\|\rho(\tau)\|_{H^s})d\tau\nonumber\\
	\le& \overline{\rho} S_{\phi_x}^{\frac{n}{4}}(t)\int_0^t ce^{-c(t-\tau)}\min\{1,\tau^{-\frac{n}{4}}\}d\tau\nonumber\\
	&+S_{\phi_x}^{\frac{n}{4}}(t)S_{U}^{\frac{n}{4}}(t) \int_0^t ce^{-c(t-\tau)}\min\{1,\tau^{-\frac{n}{2}}\}d\tau \nonumber \\
\leq & c\min\{1,t^{-\frac{n}{4}}\}\overline{\rho}S_{\phi_x}^{\frac{n}{4}}(t)\nonumber \\
&+c\min\{1,t^{-\frac{n}{2}}\}S_{\phi_x}^{\frac{n}{4}}(t)S_{U}^{\frac{n}{4}}(t)\label{K_0}.
\end{align}
In order to complete our estimate, we need to study the contribution of the diffusive part of the hyperbolic Green function. Since we are interested in the slowest decay estimate of the solution $U$, we focus on the first component: 
\begin{align}
\int_0^t & \|K(t-\tau) h(U+\overline{U},\nabla \phi)(\tau)\|_{H^s}d\tau\nonumber\\
\leq & \int_0^t   \sum_{j=1}^{n} \left\| K_{1j+1}(t-\tau)\partial_{x_j} \phi(\rho+\overline{\rho})(\tau) \right\|_{H^s}d\tau\nonumber\\
	\leq & c\int_0^t  \min\{1,(t-\tau)^{-\frac{n}{4}-1}\}\overline{\rho}\|\phi(\tau)\|_{L^1}d\tau\nonumber\\
&+c S_{\phi_x}^{\frac{1}{4}}(t) S_{U}^{\frac{1}{4}}(t) \int_0^t  \min\{1,(t-\tau)^{-\frac{n}{4}-\frac{1}{2}}\}\min\{1,\tau^{-\frac{n}{2}}\}d\tau \label{K_1}.
\end{align}

Thanks to (\ref{phiL1}), we deduce that
\begin{align}
c\int_0^t  \min\{1,(t-\tau)^{-\frac{n}{4}-1}\}\overline{\rho}\|\phi(\tau)\|_{L^1}d\tau
&\leq c\min\{1,t^{-\frac{n}{4}-1}\}\overline{\rho}\|\phi_0\|_{L^1}+ c \|\rho_0\|_{L^1}\overline{\rho}t^{-\frac{n}{4}}.\label{K_2}
\end{align}
In conclusion, substituting (\ref{5}), (\ref{6}), (\ref{K_0}),(\ref{K_2}), in (\ref{4}), we have
\begin{align*}
\left\|U(t)\right\|_{H^s} 	\le & c\left( \min\{1, t^{-\frac{n}{4}}\} E_{s} +  \min\{1, t^{-\frac{n}{4}}\} S_{U}^{\frac{n}{4}} E_{s+1} +  \min\{1,t^{-\nu}\}(S_{U}^{\frac{n}{4}}(t))^2\right.\nonumber\\
&+ \min\{1,t^{-\frac{n}{4}}\}\overline{\rho}S_{\phi_x}^{\frac{n}{4}}(t)+	 \min\{1,t^{-\frac{n}{2}}\}S_{\phi_x}^{\frac{n}{4}}(t)S_{U}^{\frac{n}{4}}(t) \\\nonumber
&+ \left.\min\{1,t^{-\frac{n}{4}}\}S_{\phi_x}^{\frac{n}{4}}(t)S_{U}^{\frac{n}{4}}(t)+\overline{\mu} \min\{1,t^{-\frac{n}{4}-1}\}\overline{\rho}\|\phi_0\|_{L^1}+  \|\rho_0\|_{L^1}\overline{\rho}t^{-\frac{n}{4}}\right).
\end{align*}
So we obtain 
\begin{align*}
S_{U}^{\frac{n}{4}}(t)	\le & c \left(E_{s}+ S_{U}^{\frac{n}{4}}(t) E_{s+1} +(S_{U}^{\frac{n}{4}}(t))^2+\overline{\rho}S_{\phi_x}^{\frac{n}{4}}(t)+\overline{\rho}D_s+ S_{\phi_x}^{\frac{n}{4}}(t) S_{U}^{\frac{n}{4}}(t)\right).
\end{align*}

Now, we substitute inequality (\ref{Mphi}) in the previous one, obtaining,
for $t>\delta>0$, 
\begin{align}
S_{U}^{\frac{n}{4}}(t)	\le & C(1 + S_{U}^{\frac{n}{4}}(t)+ (S_{U}^{\frac{n}{4}}(t))^2),\nonumber
\end{align}
where $C=C(E_s,D_{s+1},\overline{\rho})$.

From this inequality we deduce that, if the initial data and the perturbation $\overline{\rho}$ are sufficiently small, then  we have
\begin{align}
&	\left\|U(t)\right\|_{H^s} \le C \min \{1, t^{-\frac{n}{4}}\} , \nonumber\\ \nonumber\\ 
&	\left\|\phi(t)\right\|_{H^{s+1}} \le C \min \{1, t^{-\frac{n}{4}}\} . \nonumber
\end{align}
\end{proof}

\subsection{$L^\infty$ Estimates of the Solution}
We now estimate the $L^\infty$-norm of solutions to the system (\ref{sistAs}). As done before, we define the functional
\begin{align}
	R_w^{\alpha}(t):= \sup_{0 \le \tau \le t} \left\{ \max \left\{1, \tau^{\alpha}\right\} \left\|w(\tau)\right\|_{L^\infty}\right\},\nonumber
\end{align}
and 
\begin{align}
	E_s := \max \left\{ \left\|U_0\right\|_{L^1}, \, \left\|U_0\right\|_{H^s} \right\}, \qquad\quad D_s := \max \left\{ \left\|\phi_0\right\|_{L^1}, \, \left\|\phi_0\right\|_{H^s} \right\}. \nonumber
\end{align}
We want to prove the following theorem 
\begin{theorem}\label{Deca_Linfty}
Let $(U,\phi)$ be a global solution to system (\ref{sistAs}), with initial conditions
$$
U(x,0)=U_0(x),\qquad \phi(x,0)=\phi_0(x),
$$
with
$$
U_0\in H^{s+1}(\R^n)\cap L^1(\R^n), \qquad \quad \phi_0\in H^{s+1}(\R^n)\cap L^1(\R^n), \qquad \, \mbox{ for } s=\left[\frac{n}{2}\right]+2.
$$
Then the following decay estimates hold:
$$
\|U(t)\|_{L^\infty}\leq \min\{1, t^{-\frac{n}{4}}\}C(E_{s},D_{s+1}, \overline{\rho}), \qquad \|\phi(t)\|_{L^\infty}\leq \min\{1, t^{-\frac{n}{4}}\}C(E_{s},D_{s+1}, \overline{\rho}).
$$
\end{theorem}

\begin{proof}
Proceeding as done before, we obtain  $L^\infty $ estimates for $\phi$ and $\nabla\phi$. First of all we show that
\begin{align*}
\|\phi(t)\|_{L^\infty}
\leq& c e^{-bt}\|\phi_0\|_{L^\infty}+c \int_0^t e^{-b(t-\tau)}\|a\rho(\tau)\|_{L^\infty} d\tau\\
\leq& c e^{-bt}\|\phi_0\|_{L^\infty}+ c R_{U}^{\frac{n}{4}}(t)\int_0^t e^{-b(t-\tau)} \min\{1,\tau^{-\frac{n}{4}}\}d\tau,
\end{align*}
which yields
\begin{equation}
\|\phi(t)\|_{L^\infty}\leq c \left(e^{- bt}\|\phi_0\|_{L^\infty}+\min\{1,t^{-\frac{n}{4}}\}R_{U}^{\frac{n}{4}}(t)\right).\nonumber
\end{equation}
In a similar way, we get
\begin{equation}
\|\nabla\phi(t)\|_{L^\infty}\leq c\left(e^{- bt}\|\nabla \phi_0\|_{L^\infty}+\min\{1,|t-1|^{-\frac{n}{4}}\}R_{U}^{\frac{n}{4}}(t)\right).\nonumber
\end{equation}
This means that
\begin{align}
R_{\phi}^{\frac{n}{4}}&\leq C(D_{s}+R_{U}^{\frac{n}{4}}(t)), \label{phiLinf}\\
R_{\phi_x}^{\frac{n}{4}}&\leq C(D_{s+1}+R_{U}^{\frac{n}{4}}(t)).\label{phixLinf}
\end{align}

Now let us consider the solution of our system written in the form (\ref{mod4}), i.e.
\begin{align}
U(t)=&\Gamma^h(t) *U_0+ \sum_{j=1}^n \int_0^t \partial_{x_j} \Gamma^h (t-\tau)* \left[f_j'(\overline{U})U(\tau)-f_j(U(\tau))\right]d\tau\nonumber \\ &+\int_0^t \Gamma^h (t-\tau)*h(U+\overline{U},\nabla \phi)(\tau)d\tau. \nonumber	
\end{align}
Thanks to the decomposition of the Green function, we estimate the $L^\infty$-norm of $U$ in the following way
\begin{align}
\|U(t)\|_{L^\infty} 
 \le & c \min\{1,t^{- \frac{n}{2}}\}\|U_0\|_{L^1}+c e^{-ct}\|U_0\|_{H^s} \nonumber \\
 &+ \int_0^t\min \{ 1, (t-\tau)^{-\frac{n}{2}-\frac{1}{2}}\} 	  \sum_{j=1}^{n} \left\|U^2 r_j(U)(\tau)\right\|_{L^1} d\tau  \nonumber\\
&+   \int_0^t e^{-c(t-\tau)} \sum_{j=1}^{n}   \left\| \partial_{x_j}( U^2 r_j(U)(\tau)) \right\|_{H^{s}} d \tau\nonumber \\
& + \int_0^t \|\Gamma^h (t-\tau)h(U+\overline{U},\nabla\phi)(\tau)\|_{L^\infty}d \tau . \label{7}
\end{align}
The third term in the r.h.s. of (\ref{7}) is estimated as:
\begin{align}
 \int_0^t\min \{ 1, (t-\tau)^{-\frac{n}{2}-\frac{1}{2}}\} 	\sum_{j=1}^{n}  \left\|U^2 r_j(U)(\tau)\right\|_{L^1} d\tau 
&	\le  c \min\{1,t^{- \frac{n}{2}}\} E_{s}^2. \nonumber
\end{align}
While the next term in (\ref{7}) can be estimated as
\begin{align}
 \sum_{j=1}^n\int_0^t e^{-c(t-\tau)}  \left\| \partial_{x_j}( U^2 r_j(U)(\tau)) \right\|_{H^{s}} d\tau 
 & \le c R_{U}^{\frac{n}{4}}(t)E_{s+1} \int_0^t e^{-c(t-\tau)}  \min \{1, \tau^{-\frac{n}{4}}\}d\tau \nonumber\\
 & \le  c  \min\{1, t^{-\frac{n}{4}}\} R_{U}^{\frac{n}{4}}(t)E_{s+1}, \nonumber
 \end{align}
where we have used the inequality (\ref{eq4}). 

Proceeding in a similar way, we control the last term in (\ref{7}) as follows
\begin{align*}
\int_0^t \|\Gamma^h (t-\tau)*h(U+\overline{U},\nabla \phi)(\tau)\|_{L^\infty}d\tau\leq &\int_0^t \|K(t-\tau) h(U+\overline{U},\nabla\phi)(\tau)\|_{L^\infty}d\tau\\
&+\int_0^t\|\mathcal{K}(t-\tau) h(U+\overline{U},\nabla\phi)(\tau)\|_{L^\infty}d\tau.
\end{align*}
Let us start from the second integral on the right hand side:
\begin{align}
	\int_0^t \|\mathcal{K}(t-\tau) h(U+\overline{U},\nabla\phi)(\tau)\|_{L^\infty}d\tau 
	\leq& c \min\{1,t^{-\frac{n}{4}}\} \overline{\rho} S_{\phi_x}^{\frac{n}{4}}(t) \nonumber \\
	&+c\min\{1,t^{-\frac{n}{2}}\} S_{\phi_x}^{\frac{n}{4}}(t) S_{U}^{\frac{n}{4}}(t), 	\nonumber
\end{align}	
thanks to Lemma 5.2 of \cite{BiHaNa} . \\
Now we need to estimate the contributions of the diffusive part of the hyperbolic Green function 
\begin{align}
\int_0^t &\|K(t-\tau) h(U+\overline{U},\nabla \phi)(\tau)\|_{L^\infty}d\tau\nonumber\\
\leq & \int_0^t\sum_{j=1}^{n}\|K_{1j+1}(t-\tau)\rho(\tau)\partial_{x_j}\phi(\tau)\|_{L^\infty}d\tau 
+\int_0^t\sum\limits_{j=1}^{n}\|\partial_{x_j} K_{1j+1}(t-\tau)\overline{\rho}\phi(\tau)\|_{L^\infty}d\tau \nonumber\\
\leq &  S_{U}^{\frac{n}{4}}(t)S_{\phi_x}^{\frac{n}{4}}(t)\int_0^t\min\{1, (t-\tau)^{-\frac{n}{2}-\frac{1}{2}}\}\min\{1, {\tau}^{-\frac{n}{2}}\} d \tau \nonumber \\
&+ c  \int_0^t\min\{1, (t-\tau)^{-\frac{n}{2}-1}\}\overline{\rho}e^{-b \tau}d\tau + c \int_0^t\min\{1, (t-\tau)^{-\frac{n}{2}-1}\}\overline{\rho}\|\rho_0\|_{L^1}d\tau \nonumber \\
\leq& \min\{1,t^{-\frac{n}{2}}\} S_{U}^{\frac{n}{4}}(t)S_{\phi_x}^{\frac{n}{4}}(t)+c\left( \min\{1, {t}^{-\frac{n}{2}-1}\}\overline{\rho}+ t^{-\frac{n}{2}}\overline{\rho}\|\rho_0\|_{L^1}\right)\nonumber.
\end{align}

In conclusion, we obtain 
\begin{align}
	\left\|U(t)\right\|_{L^\infty} \le & c \left[\min\{1,t^{- \frac{n}{2}}\}\left\|U_0\right\|_{L^1}+ e^{-ct}\left\|U_0\right\|_{H^{s}} + \min\{1,t^{-\frac{n}{2}}\} E^2_{s}\right.\nonumber \\	
	&+\min\{1, t^{-\frac{1}{2}}\} R_{U}^{\frac{n}{4}}(t)E_{s+1}+\min\{1,t^{-\frac{n}{4}}\}\overline{\rho}S_{\phi_x}^{\frac{n}{4}}(t)
\nonumber \\
	&\left.	+\min\{1,t^{-\frac{n}{2}}\}S_{\phi_x}^{\frac{n}{4}}(t)S_{U}^{\frac{n}{4}}(t)+\overline{\mu}(\min\{1, {t}^{-\frac{n}{2}-1}\}\overline{\rho}+  t^{-\frac{n}{2}}\overline{\rho}\|\rho_0\|_{L^1})\right].\nonumber
\end{align}

Let us recall that  $S_U^{\frac{n}{4}}, \,S_{\phi_x}^{\frac{n}{4}}\leq C$. Then, substituting  inequalities (\ref{phiLinf}) and (\ref{phixLinf}) in the previous one and multiplying by $\max\{1,t^{\frac{n}{4}}\}$, we obtain
\begin{align*}
	R_{U}^{\frac{n}{4}}(t) &\le  C(1 +R_{U}^{\frac{n}{4}}(t)),\nonumber
\end{align*}
where $C=C (E_{s+1},D_{s+1},\overline{\rho})$. \\
In conclusion, if initial data and the constant state $\overline{\rho}$ are sufficiently small, 
 for the $L^\infty$-norm of the solution $(U,\phi)$ we obtain the following estimates
\begin{align}
	\left\|U(t)\right\|_{L^\infty} &\le  \min\{1, t^{-\frac{n}{4}}\}C(E_{s+1},D_{s+1}, \overline{\rho}),  \nonumber 		\\ \nonumber\\
	\left\|\phi(t)\right\|_{L^\infty} &\le  \min\{1, t^{-\frac{n}{4}}\}C(E_{s+1},D_{s+1}, \overline{\rho}).  \nonumber 		
\end{align}
\end{proof}

\section*{Acknowledgments} 
We would like to thank R. Natalini for several helpful discussions, for spending time reading and revising the manuscript carefully, and thus considerably improving the presentation. This work has been partially supported by a PRIN project 2007 on ``Sistemi iperbolici non lineari: perturbazioni singolari, comportamento asintotico e applicazioni''.

\bibliographystyle{plain}

\bibliography{bib_tesi1}

\end{document}